\documentclass[reqno]{amsart}

\usepackage{allneeded}

\allowdisplaybreaks[3]

\newcommand{\Ham}{\mathrm{Ham}}
\newcommand{\zz}{\bm{z}}
\newcommand{\ZZ}{\bm{Z}}
\newcommand{\HH}{\mathcal{H}}

\newcommand{\K}{\mathds{K}}
\newcommand{\N}{\mathds{N}}
\newcommand{\Z}{\mathds{Z}}
\newcommand{\R}{\mathds{R}}
\newcommand{\C}{\mathds{C}}

\newcommand{\PP}{\mathcal{P}}
\newcommand{\VV}{\mathds{V}}
\newcommand{\EE}{\mathds{E}}
\newcommand{\FF}{\mathds{F}}
\newcommand{\GG}{\mathds{G}}
\newcommand{\II}{\mathds{I}}
\newcommand{\WW}{\mathds{W}}
\newcommand{\fix}{\mathrm{fix}}

\newcommand{\mas}{\mathrm{mas}}
\newcommand{\comas}{\mathrm{comas}}
\newcommand{\ind}{\mathrm{ind}}
\newcommand{\SSS}{\mathrm{S}}
\newcommand{\coSSS}{\mathrm{coS}}
\newcommand{\avind}{\overline{\mathrm{ind}}}
\newcommand{\avcoind}{\overline{\mathrm{coind}}}
\newcommand{\avmas}{\overline{\mathrm{mas}}}
\newcommand{\avcomas}{\overline{\mathrm{comas}}}
\newcommand{\coind}{\mathrm{coind}}
\newcommand{\nul}{\mathrm{nul}}
\newcommand{\Sp}{\mathrm{Sp}}

\newcommand{\Tan}{\mathrm{T}}
\newcommand{\diff}{\mathrm{d}}

\begin{document}

\title{The Morse index of Chaperon's generating families}

\author{Marco Mazzucchelli}
\address{CNRS and \'Ecole Normale Sup\'erieure de Lyon, UMPA\newline\indent  69364 Lyon Cedex 07, France}
\email{marco.mazzucchelli@ens-lyon.fr}

\date{July 29, 2015}


\thanks{The author is partially supported by the ANR projects WKBHJ (ANR-12-BS01-0020) and COSPIN (ANR-13-JS01-0008-01).}

\begin{abstract}
This is an expository paper devoted to the Morse index of Chaperon's generating families of Hamiltonian diffeomorphisms. After reviewing the construction of such generating families, we present Bott's iteration theory in this setting: we study how the Morse index of a critical point corresponding to an iterated periodic orbit depends on the order of iteration of the orbit. We also investigate the precise dependence of the Morse index from the choice of the generating family associated to a given Hamiltonian diffeomorphism, which will allow to see the Morse index as a Maslov index for the linearized Hamiltonian flow in the symplectic group. We will conclude the survey with a proof that the classical Morse index from Tonelli Lagrangian dynamics coincides with the Maslov index.

\tableofcontents
\end{abstract}

\maketitle

\section{Introduction}

\subsection{Chaperon's generating family}

Generating families are classical objects that describe Hamiltonian diffeomorphisms of symplectic Euclidean spaces\footnote{More generally, generating families describe certain Lagrangian submanifolds, those who are images of the zero-section under a Hamiltonian diffeomorphism, of cotangent bundles. This more general notion originated from the work of H\"ormander \cite{Hormander:Fourier_integral_operators_I}, but was introduced in symplectic topology by Sikorav \cite{Sikorav:Problemes_d_intersections_et_de_points_fixes_en_geometrie_hamiltonienne} and further studied by many other authors.}. Consider a Hamiltonian diffeomorphism $\phi_0$ of the standard symplectic $(\R^{2d},\omega=\diff x\wedge\diff y)$. The graph of $\phi_0$ is a Lagrangian submanifold of the product $\R^{2d}\times\R^{2d}$ equipped with the symplectic form $(-\omega)\oplus\omega$. The graph of the identity diffeomorphism on $\R^{2d}$ is the diagonal subspace of $\R^{2d}\times\R^{2d}$, and the fixed points of $\phi_0$ correspond to the intersection points of its graph with the diagonal. Let us translate this picture on the cotangent bundle $\Tan^*\R^{2d}$, which is equipped with the canonical symplectic structure given by minus the exterior derivative of the Liouville form $\lambda=p\,\diff q$ (here $q$ and $p$ are the variables on the base and on the fiber respectively). We choose a symplectomorphism $(\R^{2d}\times\R^{2d},(-\omega)\oplus\omega)\to(\Tan^*\R^{2d},-\diff\lambda)$ that sends the diagonal subspace to the zero-section. In this survey, we will employ the following one:
\begin{align*}
(x_0,y_0,x_1,y_1) \mapsto (\underbrace{\big.x_1,y_0}_q,\underbrace{\big.y_1-y_0,x_0-x_1}_p).
\end{align*}
The image of the graph of $\phi_0$ under this symplectomorphism is a Lagrangian submanifold $L_0$. Assume now that $L_0$ is a section of the cotangent bundle, that is, the graph of a one-form $\mu_0$ on the base $\R^{2d}$. This is always verified provided $\phi_0$ is sufficiently close to the identity in the $C^1$-topology, or more generally whenever $\phi_0$ admits an associated diffeomorphism $\psi_0:\R^{2d}\to\R^{2d}$ such that $\phi_0(x_0,y_0)=(x_1,y_1)$ if and only if $\psi_0(x_1,y_0)=(x_0,y_1)$. Lagrangian sections of cotangent bundles are precisely the graphs of closed one-forms on the base (we refer the reader to \cite{Hofer_Zehnder:Symplectic_invariants_and_Hamiltonian_dynamics, McDuff_Salamon:Introduction_to_symplectic_topology} for this and other background results from symplectic geometry). Therefore, the one-form $\mu_0$ must be exact, i.e. $\mu_0=\diff f_0$. We say that $f_0:\R^{2d}\to\R$ is a \textbf{generating function} for the Hamiltonian diffeomorphism $\phi_0$. The explicit way $f_0$ determines $\phi_0$ is the following:
\begin{align*}
\phi_0(x_0,y_0)=(x_1,y_1)
\qquad\mbox{if and only if}\qquad
\left\{
  \begin{array}{l}
    x_1-x_0=-\partial_y f_0(x_1,y_0), \\ 
    y_1-y_0=\partial_x f_0(x_1,y_0).
  \end{array}
\right.
\end{align*}
Not only the function $f_0$ describes the Hamiltonian diffeomorphism $\phi_0$, it also provides a variational principle for the fixed points of $\phi_0$: they are precisely the critical points of $f_0$. Notice that the generating function of a Hamiltonian diffeomorphism is unique up to an additive constant.

A general Hamiltonian diffeomorphism $\phi$ of $(\R^{2d},\omega)$ does not necessarily admit a generating function, since its associated Lagrangian submanifold $L\subset\Tan^*\R^{2d}$ may not be a section. However, the following construction originally due to Chaperon~\cite{Chaperon:Une_idee_du_type_geodsiques_brisees_pour_les_systmes_hamiltoniens, Chaperon:An_elementary_proof_of_the_Conley_Zehnder_theorem_in_symplectic_geometry} allows to draw a similar conclusion provided the behavior of $\phi$ at infinity is suitably controlled. For instance, assume that $\phi$ is the time-1 map of a non-autonomous Hamiltonian flow $\phi_H^t$ whose associated Hamiltonian $H_t:\R^{2d}\to\R$ has $C^2$-norm uniformly bounded in $t\in[0,1]$ by a finite constant (this condition can be weakened). By means of this flow, we can factorize $\phi$ as
\begin{align*}
\phi=\phi_{k-1}\circ...\circ\phi_0,
\end{align*}
where each factor is given by $\phi_j:=\phi_H^{(j+1)/k}\circ(\phi_H^{j/k})^{-1}$. As we increase the number $k\in\N$ of factors, each $\phi_j$ becomes closer and closer to the identity in the $C^{1}$ topology. In particular, for $k$ large enough, each factor $\phi_j$ is described by a generating function $f_j:\R^{2d}\to\R$ as explained in the previous paragraph. Chaperon's brilliant idea was to combine these functions together in a suitable way, in order to obtain a function defined on a larger space that defines the original $\phi$. This function $F:\R^{2d}\times\R^{2d(k-1)}\to\R$ has the form
\begin{align}\label{e:Chaperon_generating_family}
F(x_k,y_0,\zz)
=
\sum_{j\in\Z_k}
\Big(\langle y_j,x_{j+1}-x_j\rangle + f_j(x_{j+1},y_j) \Big),
\end{align}
where $\zz=(z_1,...,z_{k-1})$ and $z_j=(x_j,y_j)$. A straightforward computation shows that 
\begin{gather*}
\left\{
  \begin{array}{l}
    \phi_0(x_0,y_0)=z_1, \\ 
    \phi_1(z_1)=z_2, \\ 
    \vdots \\ 
    \phi_{k-2}(z_{k-2})=z_{k-1}, \\ 
    \phi_{k-1}(z_{k-1})=(x_k,y_k),
  \end{array}
\right.\\
\Bigg.\mbox{ if and only if }\\
\left\{
  \begin{array}{l}
    x_k-x_0=-\partial_{y_0}F(x_k,y_0,\zz), \\ 
    y_k-y_0=\partial_{x_k}F(x_k,y_0,\zz),\\
    0=\partial_{\zz}F(x_k,y_0,\zz).
  \end{array}
\right.
\end{gather*}
As before, the function $F$ provides a variational principle for the fixed points of $\phi$: the vector $(x_k,y_0,x_1,y_1,...,x_{k-1},y_{k-1})$ is a critical point of $F$ if and only if $\phi_j(x_j,y_j)=(x_{j+1},y_{j+1})$ for all cyclic indices $j\in\Z_k$. We say that $F$ is a \textbf{generating family} for the Hamiltonian diffeomorphism $\phi$, associated to its factorization $\phi_{k-1}\circ...\circ\phi_0$.  Notice that a generating family becomes a simple generating function if the parameter $k$ is equal to 1. In the following, since we will employ generating families only in order to use their variational principle, we will write $x_0$ for $x_k$ in their expression.

Let us have a closer look at Chaperon's construction in the special case where the Hamiltonian diffeomorphism $\phi$ is linear, that is, when $\phi(z)=Pz$ for some symplectic matrix $P\in\Sp(2d)$. Since the symplectic group $\Sp(2d)$ is connected, we can find a continuous path $\Gamma:[0,1]\to\Sp(2d)$ joining the identity $\Gamma(0)=I$ with $\Gamma(1)=P$. This allows to build a factorization $\phi=\phi_{k-1}\circ...\circ\phi_0$, where each factor is the linear Hamiltonian diffeomorphism $\phi_j(z)=P_jz$ associated to the symplectic matrix 
\begin{align*}
P_j=\Gamma(\tfrac{j+1}{k})\Gamma(\tfrac{j}{k})^{-1}\in\Sp(2d). 
\end{align*}
Since $\phi_j$ is linear, there is a canonical way to normalize its generating function $f_j:\R^{2d}\to\R$ so that it becomes a quadratic function of the form
\begin{align*}
f_j(X_{j+1},Y_j)=
\tfrac12\langle A_j X_{j+1},X_{j+1}\rangle
+
\langle B_j X_{j+1},Y_j\rangle
+
\tfrac12\langle C_j Y_j,Y_j\rangle,
\end{align*}
where $A_j$, $B_j$, and $C_j$ are (small) $dk\times dk$ matrices, $A_j$ and $C_j$ being symmetric. This readily implies that the generating family $F:\R^{2dk}\to\R$ given by the expression~\eqref{e:Chaperon_generating_family} is a quadratic function as well, which we write as
\begin{align*}
F(\ZZ)=\tfrac12\langle H\ZZ,\ZZ\rangle
\end{align*}
for a suitable $2dk\times2dk$ symmetric matrix $H$.

\subsection{Morse indices}
Let $\phi$ be a Hamiltonian diffeomorphism of $\R^{2d}$ described by the generating family $F$ of equation~\eqref{e:Chaperon_generating_family}. Let $z_0$ be a fixed point of $\phi$, so that, if we set $z_j:=\phi_{j-1}(z_{j-1})$ for all $j=1,...,k-1$, we have a corresponding critical point $\zz=(z_0,...,z_{k-1})$ of the generating family $F$. We are interested in the Morse indices of $F$ at $\zz$, which are defined as follows. The \textbf{Morse index} $\ind(\zz)$ is the number of negative eigenvalues of the Hessian of $F$ at $\zz$ counted with multiplicity, that is, the dimension of a maximal subspace of $\R^{2dk}$ where such Hessian is negative definite. Analogously, the \textbf{Morse coindex} $\coind(\zz)$ is the number of positive eigenvalues counted with multiplicity, and finally the \textbf{nullity} $\nul(\zz)$ is the dimension of the kernel of the Hessian of $F$ at $\zz$. Notice that
\begin{align*}
\ind(\zz) + \coind(\zz) + \nul(\zz)=2dkp.
\end{align*}
In order to study these indices, let us first have a look at the Hessian of $F_p$ at $\zz$. We denote by $H(\zz)$ the symmetric $2dk\times 2dk$ matrix such that\begin{align*}
 \mathrm{Hess}F(\zz)[\ZZ,\ZZ']=\langle H(\zz)\ZZ,\ZZ' \rangle,\qquad\forall \ZZ,\ZZ'\in\R^{2dk}.
\end{align*}
Given any vector $\ZZ=(Z_1,...,Z_{k-1})\in\R^{2dk}$, its image $\ZZ':=H(\zz)\ZZ$ is given by
\begin{equation}\label{e:Hessian_at_arbitrary_critical_point}
\begin{split}
X_j' & =Y_{j-1}-Y_{j} + A_{j-1}(\zz) X_j + B_{j-1}(\zz)^T Y_{j-1},\\
Y_j' & =X_{j+1}-X_{j} + B_{j}(\zz) X_{j+1} + C_{j}(\zz) Y_{j}. 
\end{split}
\end{equation}
Here, we have adopted the common notation $Z_j=(X_j,Y_j)$ and $Z_j'=(X_j,Y_j)$. Moreover, as before, the index $j$ must be understood as an element of the cyclic group $\Z_{k}$, and we have set 
\begin{equation}
\begin{split}\label{e:blocks_of_gf_matrix}
A_{j}(\zz):=\partial_{xx}f_{j}(x_{j+1},y_{j}),\\ 
B_{j}(\zz):=\partial_{xy}f_{j}(x_{j+1},y_{j}),\\
C_{j}(\zz):=\partial_{yy}f_{j}(x_{j+1},y_{j}). 
\end{split} 
\end{equation}
From now on, we will  assume that the parameter $k$ is large enough, so that the norms of the matrices $A_j(\zz)$, $B_j(\zz)$, and $C_j(\zz)$ are bounded from above by some $\epsilon<1$.

\begin{rem}\label{r:quadratic_generating_families}
The quadratic function 
\begin{align*}
\tilde f_j(X_{j+1},Y_j)
=
\tfrac12\langle A_{j}(\zz)X_{j+1},X_{j+1}\rangle 
+
\langle B_{j}(\zz)X_{j+1},Y_{j}\rangle 
+
\tfrac12\langle C_{j}(\zz)Y_{j},Y_{j}\rangle
\end{align*}
is the generating function for the linearized map $\diff\phi_j(z_j)$. Therefore, the quadratic function $\tilde F:\R^{2dk}\to\R$ given by 
\begin{align*}
\tilde F(\ZZ)=
\tfrac12\langle H(\zz)\ZZ,\ZZ\rangle
= 
\sum_{j\in\Z_k}
\Big(
\langle Y_j,X_{j+1}-X_j\rangle + \tilde f_j(X_{j+1},Y_j)
\Big).
\end{align*}
is the generating family of the linearized map $\diff\phi(z_0)$ associated to his factorization $\diff\phi_{k-1}(z_{k-1})\circ...\circ\diff\phi_0(z_0)$.
\hfill\qed
\end{rem}

In the context of convex Hamiltonian systems, for instance in the study of closed geodesics in Riemannian manifolds, it is well known that the classical Lagrangian action functional has finite Morse indices (we will discuss this further in Section~\ref{s:Lagrangian}). Even more remarkably, there are closed geodesics that have Morse index zero when they are iterated any number of times, for instance in hyperbolic Riemannian manifolds. On the contrary, the Hamiltonian action functional has always infinite Morse indices at his critical points. Since our generating family $F$ can be considered a finite dimensional approximation of the Hamiltonian action functional, the unboundedness of the Hamiltonian Morse indices is reflected  by the fact that the Morse indices of $F$ tend to be large. For instance, if the Hamiltonian diffeomorphism $\phi$ we started with were the identity, we could choose  $f_0=...=f_{k-1}\equiv0$; the function $F$ would then be a degenerate quadratic form with Morse index and coindex both equal to $d(k-1)$. In general, we have at least the following lower bounds.

\begin{prop}\label{p:Morse_idx_always_large}
For all critical points $\zz$ of $F$, we have 
\[\min\{\ind(\zz),\coind(\zz)\}\geq d\lfloor k/2\rfloor.\]
\end{prop}
\begin{proof}
Consider the vector subspace of $(\R^{2d})^{k}$ given by
\begin{align*}
\VV:=\{ \ZZ=(Z_0,...,Z_{k-1})\in \R^{2dk}\ |\ Z_j=0 \ \forall j\ \mathrm{even},\ \ Y_h=X_h \ \forall h\ \mathrm{odd}\}.
\end{align*}
By~\eqref{e:Hessian_at_arbitrary_critical_point}, for all $\ZZ\in \VV$ we have
\begin{align*}
\langle H(\zz) \ZZ,\ZZ\rangle 
&=
\sum_{j\ \mathrm{odd}}
\Big(
-|X_j|^2 - |Y_j|^2
+ \langle A_{j-1} X_j,X_j \rangle 
+ \langle C_{j} Y_{j} ,Y_j \rangle 
\Big)\\
&\leq
\sum_{j\ \mathrm{odd}}
(\epsilon-1)(|X_j|^2 + |Y_j|^2)
\\
&=\underbrace{(\epsilon-1)}_{<0}|\ZZ|^2.
\end{align*}
This shows that the Hessian of $F$ at $\zz$ is negative definite on $\VV$, and in particular $\ind(\zz)\geq \dim \VV= d\lfloor k/2\rfloor$. By an analogous computation, the Hessian of $F$ at $\zz$ is positive definite on
\begin{align*}
\WW  :=\{ \ZZ\in \R^{2dk}\ |\ Z_j=0 \ \forall j\ \mathrm{even},\ \ Y_h=-X_h \ \forall h\ \mathrm{odd}\},
\end{align*}
which implies $\coind(\zz)\geq \dim \WW = d\lfloor k/2\rfloor$.
\end{proof}

Studying the properties of the Morse indices of generating families is tremendously important for the applications to the existence and multiplicity of periodic orbits of Hamiltonian systems. Indeed, minimax methods from non-linear analysis allow to find critical points of a generating family with almost prescribed indices. More precisely, a minimax scheme of dimension $n$, such as a minimax over the family of representative of an homology or homotopy class of degree $n$, may only converge to critical points with Morse index less than or equal to $n$ and Morse index plus nullity larger than or equal to $n$. Suppose that we are interested in the periodic points of a Hamiltonian diffeomorphism $\phi$ described by a generating family $F$. The factorization $\phi=\phi_{k-1}\circ...\circ\phi_0$ employed to build $F$ can be iterated $p$ times in order to build a generating family $F_p$ for the iterated Hamiltonian diffeomorphism $\phi^p$. A fixed point $z_0$ of $\phi$ gives a critical point $\zz=(z_0,...,z_{k-1})$ of the generating family $F$, and its $p$-th fold juxtaposition $\zz^p=(\zz,...,\zz)$ gives a critical point of the generating family $F_p$. 

Now, assume that one can setup a minimax scheme with every function $F_p$ that produces a critical point $\zz_p$ with Morse index $i_p=\ind(\zz_p)$, coindex $c_p=\coind(\zz_p)$, and nullity $n_p=\nul(\zz_p)$. The natural question to ask is whether the family of critical points $\{\zz_p\ |\ p\in\N\}$ corresponds to infinitely many distinct periodic points of $\phi$. As we just saw, the answer in general is no: in the worst case, all the critical points $\zz_p$ may be of the form $\zz^p=(\zz,...,\zz)$ and thus correspond to the same fixed point $z_0$. One way to address this question is to study the admissible behavior of the function  $p\mapsto (\ind(\zz^p),\coind(\zz^p),\nul(\zz^p))$ that associate to a period $p$ the indices of the critical points of $F_p$ corresponding to a fixed point $z_0$ of $\phi$. In the more special setting of Tonelli Lagrangian systems (c.f.\ Section~\ref{s:Lagrangian}), this idea goes back to the work of Hedlund \cite{Hedlund:Poincare_s_rotation_number_and_Morse_s_type_number} and Morse-Pitcher \cite{Morse_Pitcher:On_certain_invariants_of_closed_extremals} from the 1930s,  and  was greatly developed two decades later by Bott in his celebrated paper~\cite{Bott:On_the_iteration_of_closed_geodesics_and_the_Sturm_intersection_theory}.
 If  the sequence of indices $\{(i_p,c_p,n_p)\ |\ p\in\N\}$ provided by the minimax schemes does not have an admissible behavior, one can immediately conclude that the family of critical points $\{\zz_p\ |\ p\in\N\}$ cannot correspond to a single fixed point of $\phi$. Sometimes, this argument can be pushed further to show that such family of critical points cannot correspond to a finite set of periodic points of $\phi$, and thus infer that $\phi$ possesses infinitely many periodic points.

\subsection{Organization of the paper}

In Section~\ref{s:Bott} we will present the aforementioned Bott's iteration theory in the general setting of generating families. We will not provide applications of this theory, but we will mention some of them in the last Subsection~\ref{ss:Bott_biblio_remarks}. In Section~\ref{s:Maslov_index} we will discuss the dependence of the Morse index from the specific choice of the generating family. We will show that the Morse index can be seen as a Maslov index, a certain homotopy invariant for continuous paths in the symplectic group. In Section~\ref{s:Lagrangian} we will consider the special case of Hamiltonian diffeomorphisms generated by a non-autonomous Tonelli Hamiltonian. We will show that, in this case, the Morse indices of the generating family (or the Maslov indices of the associated symplectic paths) are related to the Morse indices of the classical Lagrangian action functional. As the reader will see, throughout the sections we will often be dealing with quadratic forms, which inevitably involves some linear algebra. In the Appendix of the paper we have collected the less standard tools from plain and symplectic linear algebra that we will need. None of the results contained in this survey is original, although some of the proofs are different form the ones available in the literature. Many authors contributed to the theory presented, and  it seems almost impossible to provide a complete and precise historical account. We will give the main references to the vast bibliography at the end of each section.

\section{Bott's iteration theory for generating families}\label{s:Bott}

\subsection{Bott indices}\label{s:Bott_indices}

Consider a Hamiltonian diffeomorphism $\phi\in\Ham(\R^{2d})$. Assume that the behavior of $\phi$ at infinity is suitably controlled, so that we have a factorization $\phi=\phi_{k-1}\circ...\circ\phi_0$ where each $\phi_j\in\Ham(\R^{2d})$ is defined by a generating function $f_j:\R^{2d}\to\R$. For each period $p\in\N$, the iterated diffeomorphism $\phi^p$ is defined by the generating family $F_p:\R^{2dkp}\to\R$ given by
\begin{align*}
F_p(z_0,...,z_{kp-1})
=
\sum_{j\in\Z_{kp}}
\Big(
\langle y_j,x_{j+1}-x_j\rangle
+
f_{j\,\mathrm{mod}\,k}(x_{j+1},y_j)
\Big),
\end{align*}
where as usual we have adopted the notation $z_j=(x_j,y_j)\in\R^{2d}$. Consider  a fixed point $z_0$ of $\phi$, with associated critical point $\zz=(z_0,...,z_{k-1})$ of $F_1$. For all periods $p\in\N$, the critical point of $F_p$ corresponding to the $p$-periodic orbit of $\phi$ starting at $z_0$ is given by $\zz^p=(\zz,...,\zz)$. Let $H_p=H_p(\zz^p)$ be the $2dkp\times2dkp$ symmetric matrix associated to the Hessian of $F_p$ at $\zz^p$, i.e.
\begin{align*}
 \mathrm{Hess}F_p(\zz^p)[\ZZ',\ZZ'']=\langle H_p \ZZ',\ZZ'' \rangle.
\end{align*}
Due to the special form of our critical point $\zz^p$, an image $\ZZ':=H_p\ZZ$ is defined by
\begin{equation}\label{e:Hessian}
\begin{split}
X_j' & =Y_{j-1}-Y_{j} + A_{j-1\,\mathrm{mod}\,k} X_j + B_{j-1\,\mathrm{mod}\,k}^T Y_{j-1},\\
Y_j' & =X_{j+1}-X_{j} + B_{j\,\mathrm{mod}\,k} X_{j+1} + C_{j\,\mathrm{mod}\,k} Y_{j},
\end{split}
\end{equation}
where the subscript $j$ belongs to $\Z_{kp}$, and the matrices $A_j=A_j(\zz)$, $B_j=B_j(\zz)$, and $C_j=C_j(\zz)$ are defined as before in~\eqref{e:blocks_of_gf_matrix}.

We wish to investigate the behavior of the Morse indices under iteration, that is, the behavior of the functions $p\mapsto\ind(\zz^p)$, $p\mapsto\coind(\zz^p)$, and $p\mapsto\nul(\zz^p)$. For this purpose, let us interpret $H_p$ in an equivalent, but conceptually slightly different, way: we see it as a second order difference operator $\HH_p$ acting on the vector space of $kp$-periodic sequences
\begin{align*}
\VV_{p}:=\big\{ (Z_j)_{j\in\Z}\in(\R^{2d})^\Z \ |\ Z_{j+kp}=Z_j\quad \forall j\in\Z \big\}.
\end{align*}
Following Bott \cite{Bott:On_the_iteration_of_closed_geodesics_and_the_Sturm_intersection_theory}, let us complexify the setting by introducing, for every $\theta$ in the unit circle $S^1\subset\C$, the vector space of sequences
\begin{align*}
\VV_{p,\theta}:=\big\{ (Z_j)_{j\in\Z}\in(\C^{2d})^\Z \ |\ Z_{j+kp}=\theta Z_j\quad \forall j\in\Z \big\}.
\end{align*}
We equip this vector space with the Hermitian product
\begin{align*}
\langle  (Z_j)_{j\in\Z},(Z_j')_{j\in\Z}\rangle_{p,\theta} = \sum_{j=0}^{kp-1} \langle  Z_j,Z_j'\rangle=\sum_{j=0}^{kp-1} Z_j \overline{Z_j'}.
\end{align*}
We introduce the linear operator $\HH_{p,\theta}:\VV_{p,\theta}\to\VV_{p,\theta}$ given by   $\HH_{p,\theta} (Z_j)_{j\in\Z}=(Z_j')_{j\in\Z}$. Here, we have denoted $Z_j=(X_j,Y_j)$, and defined $Z_j'=(X_j',Y_j')$  by the equations~\eqref{e:Hessian}, where the subscript $j$ is now in $\Z$. The operator $\HH_{p,\theta}$ is Hermitian with respect to the product $\langle  \cdot,\cdot\rangle_{p,\theta}$, and in particular it has real spectrum. Indeed, the vector space $\VV_{p,\theta}$ is isomorphic to $(\C^{2d})^{kp}$ via the map \[(Z_j)_{j\in\Z}\mapsto(Z_0,...,Z_{kp-1}),\] which pulls back the standard Hermitian product on $\C^{2dkp}$ to $\langle  \cdot,\cdot\rangle_{p,\theta}$. By means of this isomorphism, we can see $\HH_{p,\theta}$ as the complex linear endomorphism $H_{p,\theta}$ of  $\C^{2dkp}$ given by $H_{p,\theta}\ZZ=\ZZ'$, where $X_1',...,X_{kp-1}',Y_0',...,Y_{p-2}'$ are defined as in~\eqref{e:Hessian}, while
\begin{align*}
X_0' & =\overline\theta\, Y_{kp-1} - Y_{0} + A_{k-1} X_0 + \overline\theta\, B_{k-1}^T Y_{kp-1},\\
Y_{kp-1}' & = \theta\, X_{0} - X_{kp-1} + \theta\,B_{k-1} X_{0} + C_{k-1} Y_{kp-1}. 
\end{align*}
The difference with respect to~\eqref{e:Hessian} is that there are some coefficients $\theta$ or $\overline\theta$ appearing, according to the fact that the sequences in $\VV_{p,\theta}$ are $kp$-periodic only after ``twisting'' them by $\theta$. If we see $H_{p,\theta}$ as a $2dkp\times 2dkp$ complex matrix, the above expressions readily imply that $H_{p,\theta}^*=H_{p,\theta}$. 

In the following, we will refer to $H_{p,\theta}$ as to the \textbf{$\theta$-Hessian} of $F_p$ at $\zz^p$. We generalize the Morse indices and the nullity by introducing the following \textbf{Bott indices}
\begin{align*}
\ind_{\theta}(\zz^p)&=\sum_{\lambda<0} \dim_\C\ker (\HH_{p,\theta}-\lambda I)
=\sum_{\lambda<0} \dim_\C\ker (H_{p,\theta}-\lambda I),\\
\coind_{\theta}(\zz^p)&=\sum_{\lambda > 0} \dim_\C\ker (\HH_{p,\theta}-\lambda I)
=\sum_{\lambda > 0} \dim_\C\ker (H_{p,\theta}-\lambda I),\\
\nul_{\theta}(\zz^p)&= \dim_\C\ker \HH_{p,\theta}
=\dim_\C\ker H_{p,\theta}.
\end{align*}
The usual Morse indices correspond to the case where $\theta=1$, that is, 
\begin{align*}
\ind(\zz^p)&=\ind_1(\zz^p),\\
\coind(\zz^p)&=\coind_1(\zz^p),\\
\nul(\zz^p)&=\nul_{1}(\zz^p).
\end{align*} 
The first elementary properties of the Bott indices are the following.
\begin{lem}\label{l:properties_ind_theta}
$ $
\begin{itemize}
\item[(i)] The functions $\theta\mapsto\ind_{\theta}(\zz^p)$, $\theta\mapsto\coind_{\theta}(\zz^p)$ and $\theta\mapsto\nul_{\theta}(\zz^p)$ are invariant by complex conjugation.

\item[(ii)] $\nul_\theta(\zz^p)=\dim_\C\ker (\diff\phi^p(z_0)-\theta I)$.

\item[(iii)] The functions $\theta\mapsto\ind_\theta(\zz^p)$ and $\theta\mapsto\coind_\theta(\zz^p)$ are locally constant on $S^1\setminus\sigma(\diff\phi^p(z_0))$, the complement of the set of eigenvalues of $\diff\phi^p(z_0)$ on the unit circle. Given an open interval $U\subset S^1$ such that the intersection $U\cap\sigma(\diff\phi^p(z_0))$ contains only one point $\theta$, for all $\theta'\in U\setminus\{\theta\}$ we have
\begin{align*}
\underbrace{\ind_{\theta'}(\zz^p) - \ind_\theta(\zz^p)}_{\geq0}
+ \underbrace{\coind_{\theta'}(\zz^p) - \coind_\theta(\zz^p)}_{\geq0}
=
\nul_\theta(\zz^p).
\end{align*}
\end{itemize}
\end{lem}

\begin{proof}
Point (i) is an immediate consequence of the fact that $\overline {H_{p,\theta}}=H_{p,\overline\theta}$. 

As for point~(ii), notice that a vector $\ZZ=(Z_0,...,Z_{kp-1})$ belongs to the kernel of $H_{p,\theta}$ if and only if it satisfies, for all $j=0,...,kp-2$,
\begin{align*}
X_{j+1} - X_{j} &= -B_{j\,\mathrm{mod}\,k} X_{j+1} - C_{j\,\mathrm{mod}\,k} Y_{j},\\ 
Y_{j+1} - Y_{j} &=  A_{j\,\mathrm{mod}\,k} X_{j+1} + B_{j\,\mathrm{mod}\,k}^T Y_{j},
\end{align*}
and
\begin{align*}
\theta\, X_{0} & =X_{kp-1} - B_{k-1} \theta\,X_{0} - C_{k-1} Y_{kp-1},\\
\theta\,Y_{0} &=  Y_{kp-1} + A_{k-1} \theta\,X_0 + B_{k-1}^T Y_{pk-1}.
\end{align*}
We already saw in Remark~\ref{r:quadratic_generating_families} that the quadratic function $\tilde f_j$ is the generating function for the linearized map $\diff\phi_j(z_j)$. Therefore, we can rephrase the above conditions by saying that a vector $\ZZ=(Z_0,...,Z_{kp-1})$ belongs to the kernel of $H_{p,\theta}$ if and only if $\diff\phi_{j\,\mathrm{mod}\,k}(z_{j\,\mathrm{mod}\,k})Z_j=Z_{j+1}$ for all $j=0,...,kp-2$ and $\diff\phi_{k-1}(z_{k-1})Z_{kp-1}=\theta\,Z_{0}$. The projection $\ZZ\mapsto Z_0$ is thus a diffeomorphism between the kernel of $H_{p,\theta}$ and the kernel of $\diff\phi^p(z_0)-\theta I$.

Point~(iii) is a consequence of the continuity of the function that associates to a matrix his set of eigenvalues. Let us explain this in detail. First of all, since the matrix $H_{p,\theta}$ is Hermitian, it is diagonalizable. In particular $\dim_\C\ker (H_{p,\theta}-\lambda I)$ is equal to the algebraic multiplicity of $\lambda$ as an eigenvalue of $H_{p,\theta}$ (which is understood to be zero if $\lambda$ is not an eigenvalue). Fix an arbitrary $\theta\in S^1$. For an open interval $U\subset S^1$ containing $\theta$, there exist a continuous function 
\begin{align*}
\bm\lambda=(\lambda_1,\lambda_2,...,\lambda_{2dkp}):U\to\R^{2dkp}
\end{align*}
such that, for all $\theta'\in U$, the numbers $\lambda_1(\theta'),\lambda_2(\theta'),...,\lambda_{2dkp}(\theta')$ are the eigenvalues of $H_{p,\theta'}$ repetead according to their algebraic multiplicity. In particular, we have
\begin{align*}
\ind_{\theta'}(\zz^p)&=\#\{j \ |\ \lambda_j(\theta')<0\},\\
\coind_{\theta'}(\zz^p)&=\#\{j \ |\ \lambda_j(\theta')>0\},\\
\nul_{\theta'}(\zz^p)&=\#\{j \ |\ \lambda_j(\theta')=0\}.
\end{align*}
This immediately implies that, if $\nul_{\theta}(\zz^p)=0$, the function $\theta'\mapsto \ind_{\theta'}(\zz^p)$ is constant in a neighborhood of $\theta$. 
Assume now that $\nul_\theta(\zz^p)>0$, and shrink $U$ around $\theta$ so that it does not contains other eigenvalues of $\diff\phi^p(z_0)$. In particular, the sign of each function $\lambda_j$ is locally constant on $U\setminus\{\theta\}$. Therefore, the difference $\ind_{\theta'}(\zz^p)-\ind_{\theta}(\zz^p)$ is precisely the number of subscripts $j$ such that $\lambda_j(\theta')<0$ and $\lambda_j(\theta)=0$. Analogously,  $\coind_{\theta'}(\zz^p)-\coind_{\theta}(\zz^p)$ is the number of subscripts $j$ such that $\lambda_j(\theta')>0$ and $\lambda_j(\theta)=0$. Finally, $\nul_{\theta'}(\zz^p)=0$ for all $\theta'\in U\setminus\{\theta\}$. This proves  point~(iii).
\end{proof}

As we mentioned earlier, the reason for introducing the Bott indices is that the function $\theta\mapsto\ind_{\theta}(\zz)$ alone determines the iterated index $\ind(\zz^p)$ for all periods $p\in\N$, and the same property holds for the coindices and the nullities. The precise way this works is explained by the following lemma.

\begin{lem}[Bott's formulae]\label{l:Bott_formulae}
For all $p\in\N$ and $\theta\in S^1$, we have
\begin{align*}
\nul_\theta(\zz^p)&=\sum_{\mu\in\sqrt[p]{\theta}}\nul_\mu(\zz),\\
\ind_\theta(\zz^p)&=\sum_{\mu\in\sqrt[p]{\theta}}\ind_\mu(\zz),\\
\coind_\theta(\zz^p)&=\sum_{\mu\in\sqrt[p]{\theta}}\coind_\mu(\zz).
\end{align*}
\end{lem}

\begin{proof}
The first equality follows from a general property of matrices. Indeed, by Lemma~\ref{l:properties_ind_theta}(ii), such an equality can be rewritten as
\[
\dim_\C\ker(\diff\phi^p(z_0)-\theta I) = \sum_{\mu\in\sqrt[p]\theta} \dim_\C\ker(\diff\phi(z_0)-\mu I),
\]
which follows from Proposition~\ref{p:power_matrix}.

Now, we are going to provide an argument that proves the three equalities of the lemma at once. Indeed, we will show that
\begin{align}\label{e:general_Fourier_decomposition}
\dim_\C\ker (\HH_{p,\theta}-\lambda I) = \sum_{\mu\in\sqrt[p]{\theta}} \dim_\C\ker (\HH_{1,\mu}-\lambda I),\qquad\forall\lambda\in\R.
\end{align}
For this, we need an ingredient from elementary Fourier analysis. Notice first that $\VV_{1,\mu}$ is a vector subspace of $\VV_{p,\theta}$ whenever $\mu^p=\theta$. Any sequence of complex vectors $\ZZ=(Z_j)_{j\in\Z}\in \VV_{p,\theta}$ can be  decomposed as
\begin{align}\label{e:Fourier_decomposition}
\ZZ=\sum_{\mu\in\sqrt[p]\theta} \ZZ_\mu,
\end{align}
where $\ZZ_\mu=(Z_{\mu,j})_{j\in\Z}\in \VV_{1,\mu}$ is given by
\begin{align*}
Z_{\mu,j}:=\frac1{kp} \sum_{h=0}^{kp-1} \mu^{1-h} Z_{h+j}.
\end{align*}
Given two distinct roots $\mu,\sigma\in\sqrt[p]\theta$, the corresponding vector spaces $\VV_{1,\mu}$ and $\VV_{1,\sigma}$ are orthogonal with respect to the Hermitian product $\langle\cdot,\cdot\rangle_{p,\theta}$. Indeed, if $\ZZ'\in \VV_{1,\mu}$ and $\ZZ''\in \VV_{1,\sigma}$, we have
\begin{align*}
\langle\ZZ',\ZZ''\rangle_{p,\theta}
=
\sum_{j=0}^{kp-1}
Z'_j \overline{Z''_j}
=
\sum_{j=0}^{k-1}
Z'_j \overline{Z''_j}
\underbrace{\sum_{h=0}^{p-1}(\mu \overline\sigma)^h}_{=0}=0.
\end{align*}
This readily implies that the decomposition~\eqref{e:Fourier_decomposition} is unique, and defines a $\langle\cdot,\cdot\rangle_{p,\theta}$-orthogonal splitting
\begin{align*}
\VV_{p,\theta}=\bigoplus_{\mu\in\sqrt[p]\theta} \VV_{1,\mu}.
\end{align*}
Actually, this splitting turns out to be orthogonal also with respect to the Hermitian form $\langle \HH_{p,\theta}\cdot,\cdot\rangle_{p,\theta}$. Indeed, $\HH_{p,\theta}|_{\VV_{1,\mu}}=\HH_{1,\mu}$ and, if $\ZZ'$ and $\ZZ''$ are as above, we have
\begin{align*}
\langle \HH_{p,\theta}\ZZ',\ZZ''\rangle_{p,\theta}
=
\langle \underbrace{\big.\HH_{1,\mu}\ZZ'}_{\in\VV_{1,\mu}},\ZZ''\rangle_{p,\theta}
=
0. 
\end{align*}
In particular, the $\lambda$-eigenspace of $\HH_{p,\theta}$ is the direct sum of the $\lambda$-eigenspaces of the operators $\HH_{1,\mu}$, for all $\mu\in\sqrt[p]\theta$, and equation~\eqref{e:general_Fourier_decomposition} follows.
\end{proof}

Lemmata~\ref{l:properties_ind_theta} and~\ref{l:Bott_formulae} give a clear picture of the qualitative behavior of the functions $p\mapsto\ind(\zz^p)$ and $p\mapsto\coind(\zz^p)$. In particular, they imply that the quantities
\begin{equation}\label{e:definition_average_indices}
\begin{split}
\avind(\zz)&:= \frac 1{2\pi} \int_0^{2\pi} \ind_{e^{it}}(\zz)\,\diff t,\\ \avcoind(\zz)&:= \frac 1{2\pi} \int_0^{2\pi} \coind_{e^{it}}(\zz)\,\diff t 
\end{split}
\end{equation}
are always finite, and we have
\begin{align}\label{e:average_index_as_growth_rate}
\avind(\zz)&=\lim_{p\to\infty} \frac{\ind(\zz^p)}{p},\\
\avcoind(\zz)&=\lim_{p\to\infty} \frac{\coind(\zz^p)}{p}.
\end{align}
In the following, we will refer to $\avind(\zz)$ and $\avcoind(\zz)$ respectively as to the \textbf{average Morse index} and \textbf{coindex} of the critical point $\zz$. Notice that, by the conjugacy-invariance of the function $\theta\mapsto\ind_\theta(\zz)$, in the above expressions~\eqref{e:definition_average_indices} we can replace $2\pi$ by $\pi$, that is, we can equivalently average the index functions on the upper semi-circle. Equations~\eqref{e:average_index_as_growth_rate} and Proposition~\ref{p:Morse_idx_always_large} imply that 
\begin{gather*}
dk/2 \leq \avind(\zz) \leq 2dk,\\
dk/2 \leq \avcoind(\zz) \leq 2dk.
\end{gather*}
Since $\ind(\zz^p)+\coind(\zz^p)+\nul(\zz^p)=2dkp$, we further have
\begin{align*}
\avind(\zz)+\avcoind(\zz)=2dk.
\end{align*}
Another property of the average indices that follows immediately from their definitions is that
\begin{align*}
\avind(\zz^p)&=p\,\avind(\zz),\\ 
\avcoind(\zz^p)&=p\,\avcoind(\zz).
\end{align*}

Now, we are going to find optimal bounds from the gap between the average and the actual Morse indices. Such bounds plays an essential role in the multiplicity problem for periodic points of Hamiltonian diffeomorphisms (see Section~\ref{ss:Bott_biblio_remarks}). For now, we can only deal with the non-degenerate situation (Theorem~\ref{t:iteration_inequality_nondegenerate} will be superseded by the general Theorem~\ref{t:iteration_inequalities}). We recall that $d$ is the half-dimension of the domain of our Hamiltonian diffeomorphism $\phi$.

\begin{thm}\label{t:iteration_inequality_nondegenerate}
Assume that $\zz$ is a non-degenerate critical point of $F_1$, i.e.\ $\nul(\zz)=0$. Then $|\avind(\zz)-\ind(\zz)|< d$ and $|\avcoind(\zz)-\coind(\zz)|< d$.
\end{thm}

\begin{proof}
We will provide the proof for the Morse index, the one for the coindex being identical. For any eigenvalue on the unit circle $\theta\in\sigma(\diff\phi(z_0))\cap S^1$, let $\epsilon>0$ be a small enough quantity so that $\sigma(\diff\phi(z_0))\cap S^1$ does not contain other eigenvalues with arguments in the interval $[\arg(\theta)-\epsilon,\arg(\theta)+\epsilon]$. We set $\theta^\pm:=\theta e^{\pm i\epsilon}$. For all $\mu\in S^1$ with $\mathrm{Im}(\mu)>0$, we denote by $\sigma_\mu$ the (possibly empty) set of eigenvalues of $\diff\phi(z_0)$ on the unit circle with argument in the open interval $(0,\arg(\mu))$, and we define 
\[f(\mu):=\sum_{\theta\in \sigma_\mu} \big(\ind_{\theta^+}(\zz)-\ind_{\theta^-}(\zz)\big).\]
By its definition, the function $f$ is piecewise constant. By Lemma~\ref{l:properties_ind_theta}(iii), if $\mu$ is not an eigenvalue of $\diff\phi(z_0)$, we have
\begin{align*}
\ind_{\mu}(\zz)-\ind(\zz)=  f(\mu).
\end{align*}
By integrating this equality in $\mu$ on the upper semi-circle, we obtain
\begin{align*}
\avind(\zz)-\ind(\zz) &= \frac1\pi \int_0^{\pi} \big(\ind_{e^{it}}(\zz) - \ind(\zz)\big)\,\diff t
\\
 &= \frac1\pi \int_0^{\pi} f(e^{it})\, \diff t,
\end{align*}
By the equality in Lemma~\ref{l:properties_ind_theta}(iii), for all $t\in(0,\pi)$ we can estimate
\begin{align*}
|f(e^{it})|\leq\sum_{\theta\in \sigma_{\exp(it)}}\nul_\theta(\zz) 
\leq\sum_{\theta\in S^1\cap\{\mathrm{Im}>0\}}\nul_\theta(\zz) 
\leq\frac12\sum_{\theta\in S^1}\nul_\theta(\zz) \leq d.
\end{align*}
Let $\delta>0$ be such that there is no eigenvalue of $\diff\phi(z_0)$ on the unit circle with argument in $[0,\delta]$. In particular, the function $t\mapsto f(e^{it})$ is zero on the interval $[0,\delta]$. Therefore, we conclude
\begin{align*}
|\avind(\zz)-\ind(\zz)| &= \left|\frac1\pi \int_\delta^{\pi} f(e^{it})\, \diff t\right|\\
&\leq \frac1\pi \int_\delta^{\pi} |f(e^{it})|\, \diff t\\
&\leq \frac{\pi-\delta}{\pi}\,d\\
&< d. \qedhere
\end{align*}
\end{proof}

\subsection{Splitting numbers}\label{s:splitting_numbers}

The generalization of Theorem~\ref{t:iteration_inequality_nondegenerate} to the degenerate situation requires new ingredients, which incidentally will shed some light on the dependence of the Morse index of the critical point associated to a fixed point $z_0\in\fix(\phi)$ from the specific generating family employed (this dependence will be explored further in Section~\ref{s:Maslov_index}). 

Since in this section we will work in the fixed period $p=1$, in order to ease the notation we will drop it from all appearing symbols, thus writing $H_\theta$ for the $\theta$-Hessian $H_{1,\theta}$. We will denote by $h_\theta:\C^{2dk}\times\C^{2dk}\to\C$ the Hermitian bilinear form associated to $H_\theta$, i.e.
\begin{align*}
h_\theta(\ZZ,\ZZ')=\langle H_\theta\ZZ,\ZZ'\rangle.
\end{align*}
We consider the vector subspace
\begin{align*}
\VV:=\big\{\ZZ=(Z_0,...,Z_{k-1})\in\C^{2dk}\ \big|\ X_0=0\big\},
\end{align*}
where, as before, we write $Z_j=(X_j,Y_j)$. We will reduce the computation of the inertia of $h_\theta$ to the inertia of its restrictions to $\VV$ and to its $h_\theta$-orthogonal space $\VV^{h_\theta}$ by means of Propositions~\ref{p:nul_restricted_form} and~\ref{p:index_restricted_form}, which give
\begin{align}
\label{e:ind_h_theta_restricted}
\ind(h_\theta)  =\ & \ind(h_\theta|_{\VV\times\VV}) + \ind(h_\theta|_{\VV^{h_\theta}\times\VV^{h_\theta}}) \\
\nonumber
& + \dim_\C(\VV\cap\VV^{h_\theta}) - \dim_\C(\VV\cap\ker(H_\theta)),
\\
\label{e:coind_h_theta_restricted}
\coind(h_\theta)  =\ & \coind(h_\theta|_{\VV\times\VV}) + \coind(h_\theta|_{\VV^{h_\theta}\times\VV^{h_\theta}}) \\
\nonumber
& + \dim_\C(\VV\cap\VV^{h_\theta}) - \dim_\C(\VV\cap\ker(H_\theta)),
\\
\label{e:nul_h_theta_restricted}
\nul(h_\theta)=\ &\nul(h_\theta|_{\VV^{h_\theta}\times\VV^{h_\theta}}) - \dim_\C(\VV\cap\VV^{h_\theta}) + \dim_\C(\VV\cap\ker(H_\theta)).
\end{align} 
We refer the reader to Appendix~\ref{a:restriction_quadratic_forms} for the terminology and the notation concerning Hermitian forms. The restriction of  $h_\theta$ to $\VV$ is independent of $\theta$. Indeed, for all $\ZZ,\ZZ'\in\VV$, we have
\begin{align*}
h_\theta(\ZZ,\ZZ')
= & \,
\langle \overline\theta\, Y_{k-1} - Y_{0} + A_{k-1} X_0 + \overline\theta\, B_{k-1}^T Y_{k-1},X_0'\rangle \\
 & + \langle \theta\, X_{0} - X_{k-1} + \theta\,B_{k-1} X_{0} + C_{k-1} Y_{k-1},Y_{k-1}'\rangle \\
 & + \sum_{j=1}^{k-1} \langle Y_{j-1} - Y_{j} + A_{j-1} X_j + B_{j-1}^T Y_{j-1} , X_j'\rangle \\
 & + \sum_{j=0}^{k-2} \langle X_{j+1} - X_{j} + B_{j} X_{j+1} + C_{j} Y_{j},Y_j'\rangle\\ 
 = & \,
 \langle - X_{k-1} + C_{k-1} Y_{k-1},Y_{k-1}'\rangle + \langle X_{1} + B_{0} X_{1} + C_{0} Y_{0},Y_0'\rangle  \\
 & + \sum_{j=1}^{k-1} \langle Y_{j-1} - Y_{j} + A_{j-1} X_j + B_{j-1}^T Y_{j-1} , X_j'\rangle \\
 & + \sum_{j=1}^{k-2} \langle X_{j+1} - X_{j} + B_{j} X_{j+1} + C_{j} Y_{j},Y_j'\rangle .
\end{align*}
In particular, the inertia indices $\ind(h_\theta|_{\VV\times\VV})$ and $\coind(h_\theta|_{\VV\times\VV})$ are independent of $\theta\in S^1$. The orthogonal space $\VV^{h_\theta}$ contains precisely the vectors $\ZZ\in\C^{2dk}$ such that 
\begin{align*}
 \theta\, X_{0} - X_{k-1} + \theta\,B_{k-1} X_{0} + C_{k-1} Y_{k-1} & =0,\\
 X_{j+1} - X_{j} + B_{j} X_{j+1} + C_{j} Y_{j} & = 0, &\forall j=0,...,k-2,\\
 Y_{j-1} - Y_{j} + A_{j-1} X_j + B_{j-1}^T Y_{j-1} & =0, &\forall j=1,...,k-1.
\end{align*}
This means that, if we set $P_j:=\diff\phi_j(z_j)$ for all $j=0,...,k-1$,
\begin{align*}
\VV^{h_\theta}=
\left\{
(Z_0,...,Z_{k-1})\in \C^{2dk}
\, \left|\ 
  \begin{array}{@{}l@{}}
    P_j Z_j =Z_{j+1}\quad\forall j=0,...,k-2 \vspace{5pt} \\ 
    P_{k-1} Z_{k-1} =(\theta X_0,\tilde Y_k)\mbox{ for some }\tilde Y_k\in\C^d 
  \end{array}
\right.\right\}.
\end{align*}
In particular, $\VV^{h_\theta}$ is isomorphic to $(\diff\phi(z_0)-\theta I)^{-1}(\{0\}\times\C^d)$ via the isomorphism $\ZZ\mapsto Z_0$. Therefore, its dimension is bounded as
\begin{align*}
\dim_\C \VV^{h_\theta}\leq d + \dim_\C\ker(\diff\phi(z_0)-\theta I).
\end{align*}
The intersection $\VV\cap\VV^{h_\theta}$ is equal to
\begin{align*}
\VV\cap\VV^{h_\theta}=
\left\{
(Z_0,...,Z_{k-1})\in \C^{2dk}
\, \left|\ 
  \begin{array}{@{}l@{}}
    X_0=0 \vspace{5pt} \\
    P_j Z_j =Z_{j+1}\quad\forall j=0,...,k-2 \vspace{5pt} \\ 
    P_{k-1} Z_{k-1}=(0,\tilde Y_k)\mbox{ for some }\tilde Y_k\in\C^d 
  \end{array}
\right.\right\}.
\end{align*}
In particular, it is independent of $\theta$. The intersection $\VV\cap\ker H_\theta$ is equal to
\begin{align*}
\VV\cap\ker H_\theta
=
\left\{
(Z_0,...,Z_{k-1})\in \C^{2dk}
\, \left|\ 
  \begin{array}{@{}l@{}}
    X_0=0 \vspace{5pt} \\
    P_j Z_j=Z_{j+1}\quad\forall j=0,...,k-2 \vspace{5pt} \\ 
    P_{k-1} Z_{k-1} =(0,\theta Y_0) 
  \end{array}
\right.\right\}.
\end{align*}
Therefore, the map $\ZZ\mapsto Z_0=(0,Y_0)$ is an isomorphism between $\VV\cap\ker H_\theta$ and $\ker(\diff\phi(z_0)-\theta I)\cap (\{0\}\times\C^d)$. In particular
\begin{align*}
\dim_\C(\VV\cap\ker H_\theta) = \dim_\C \big(\ker(\diff\phi(z_0)-\theta I)\cap (\{0\}\times\C^d)\big).
\end{align*}
Let us now have a look at the restriction of the Hermitian form $h_\theta$ to $\VV^{h_\theta}$. For all $\ZZ,\ZZ'\in\VV^{h_\theta}$, we have
\begin{align*}
h_\theta(\ZZ,\ZZ')
& =  \langle \overline\theta Y_{k-1} - Y_0  + \underbrace{A_{k-1} X_0 + \overline\theta B_{k-1}^T Y_{k-1}}_{\overline \theta  \tilde Y_{k} - \overline \theta  Y_{k-1}} , X_0' \rangle \\
& =  \langle \overline \theta  \tilde Y_{k} - Y_0  , X_0' \rangle \\
& = \omega( (I-\overline\theta\,\diff\phi(z_0))Z_0,Z_0' ),
\end{align*}
where $\omega$ denotes the Hermitian extension of the standard symplectic form on $\R^{2d}$, given by $\omega(Z,Z')=\langle X,Y'\rangle -\langle Y,X'\rangle$. Summing up, we have shown that $\ind(h_\theta|_{\VV\times\VV})$ and $\dim_\C(\VV\cap\VV^{h_\theta})$ are independent of $\theta$, while $\dim_\C(\VV\cap\ker(H_\theta))$ and $\ind(h_\theta|_{\VV^{h_\theta}\times\VV^{h_\theta}})$ are completely determined by the linearized map $\diff\phi(z_0)$. This, together with equations~\eqref{e:ind_h_theta_restricted} and~\eqref{e:coind_h_theta_restricted}, implies the following.

\begin{lem}\label{p:Bott_function_dependence}
The functions $\theta\mapsto\ind_\theta(\zz)=\ind(h_\theta)$ and $\theta\mapsto\coind_\theta(\zz)=\coind(h_\theta)$ are completely determined by the linearized map $P:=\diff\phi(z_0)\in\Sp(2d)$ up to additive constants.
\hfill\qed
\end{lem}

We call \textbf{splitting numbers} of the linearized map $P$ at $\theta\in S^1$ the two quantities
\begin{equation}\label{e:splitting_numbers}
\begin{split}
\SSS_P^+(\theta) & = \ind_{\theta^+}(\zz) - \ind_{\theta}(\zz)=\ind(h_{\theta^+})-\ind(h_\theta),\\
\SSS_P^-(\theta) & = \ind_{\theta^-}(\zz) - \ind_{\theta}(\zz)=\ind(h_{\theta^-})-\ind(h_\theta),
\end{split} 
\end{equation}
where $\theta^\pm=\theta e^{\pm i\epsilon}$, and $\epsilon>0$ is sufficiently small so that $\sigma(P)\cap S^1$ does not contain eigenvalues with arguments in $[\arg(\theta)-\epsilon,\arg(\theta))\cup(\arg(\theta),\arg(\theta)+\epsilon]$. Lemma~\ref{p:Bott_function_dependence} guarantees that $\SSS_P^\pm$ is a good notation: the splitting numbers only depend on the linearized map $P\in\Sp(2d)$. Namely, if $\phi'$ is another Hamiltonian diffeomorphism of $\R^{2d}$ with a fixed point $z_0'$ and the same linearized map $P=\diff\phi(z_0)=\diff\phi'(z_0')$, given a  generating family $F'$ for $\phi'$, the splitting numbers functions associated to the $\theta$-Hessian of $F'$ at the critical point corresponding to $z_0'$ are still $\SSS_P^\pm$. By replacing indices with coindices in~\eqref{e:splitting_numbers}, we can define the \textbf{cosplitting numbers} 
\begin{align*}
\coSSS_P^+(\theta) & = \coind_{\theta^+}(\zz) - \coind_{\theta}(\zz)=\coind(h_{\theta^+})-\coind(h_\theta),\\
\coSSS_P^-(\theta) & = \coind_{\theta^-}(\zz) - \coind_{\theta}(\zz)=\coind(h_{\theta^-})-\coind(h_\theta),
\end{align*}
which possess analogous properties. The equality in Lemma~\ref{l:properties_ind_theta}(iii) can be rewritten as
\begin{align}\label{e:relation_splitting_cosplitting}
\underbrace{\SSS^\pm(\theta)}_{\geq0}
+
\underbrace{\coSSS^\pm(\theta)}_{\geq0} 
= 
\dim_\C \ker(P-\theta I),
\qquad
\forall \theta\in S^1.
\end{align}

\begin{war}
Many authors in symplectic topology use a different sign convention, and thus call splitting numbers what we call cosplitting numbers. The convention adopted in a paper can be easily checked on Example~\ref{ex:splitting_numbers_shear}. See also Warning~\ref{w:convention_Maslov_index} in the next section.
\hfill\qed
\end{war}

\begin{rem}
The splitting and cosplitting numbers can be defined for any symplectic matrix $P\in\Sp(2d)$. Indeed, the symplectic group $\Sp(2d)$ is connected, and therefore the map $\phi(z)=Pz$ is a Hamiltonian diffeomorphism such that $\diff\phi(0)=P$.
\hfill\qed
\end{rem}

We will now strengthen Lemma~\ref{p:Bott_function_dependence} as follows.

\begin{lem}
The splitting and cosplitting numbers $\SSS_P^\pm(\theta)$ and $\coSSS_P^\pm(\theta)$ only depend on the conjugacy class of $P$ in the symplectic group: for all $Q\in\Sp(2d)$, we have
\begin{align*}
\SSS_{P}^\pm(\theta) & =\SSS_{QPQ^{-1}}^\pm(\theta),\\
\coSSS_{P}^\pm(\theta) & =\coSSS_{QPQ^{-1}}^\pm(\theta).
\end{align*}
\end{lem}

\begin{proof}
Since the symplectic group is connected, there exists a smooth path of symplectic matrices $Q_t\in\Sp(2d)$ such that $Q_0=I$ and $Q_1=Q$. We set $\phi^t(z):=Q_t P Q_t^{-1}z$, and we consider $t\mapsto\phi^t$ as a smooth path of Hamiltonian diffeomorphisms of $\R^{2d}$. For $k\in\N$ large enough, there exists a smooth homotopy 
\[F^t:\R^{2dk}\to\R,\qquad t\in[0,1],\]
$F^t$ being the quadratic generating family of $\phi^t$. The origin $0\in\R^{2dk}$ is the critical point of $F^t$ corresponding to the fixed point $0\in\R^{2d}$ of $\phi^t$. For all $\theta\in S^1$, we denote by $H^t_\theta$  the $\theta$-Hessian of $F^t$ at the origin, and by $h^t_\theta$ the associated Hermitian form
\begin{align*}
h^t_\theta(\ZZ,\ZZ')=\langle H^t_\theta\ZZ,\ZZ'\rangle.
\end{align*}
Notice that $H^t_\theta$ depends smoothly on $(t,\theta)\in[0,1]\times S^1$, and that 
\[\ker H^t_\theta=\ker(Q_t P Q_t^{-1}-\theta I)=\ker (Q_t(P -\theta I)Q_t^{-1}).\]
In particular, the function $t\mapsto\dim_\C\ker H^t_\theta$ is constant. This readily implies that the functions $t\mapsto\ind(h^t_\theta)$ and $t\mapsto\coind(h^t_\theta)$ are constant as well, and therefore
\begin{align*}
\SSS_{P}^\pm(\theta)
& =  
\ind(h_{\theta^\pm}^0)-\ind(h_\theta^0) =
\ind(h_{\theta^\pm}^1)-\ind(h_\theta^1) =
\SSS_{QPQ^{-1}}^\pm(\theta),\\
\coSSS_{P}^\pm(\theta)
& =  
\coind(h_{\theta^\pm}^0)-\coind(h_\theta^0) =
\coind(h_{\theta^\pm}^1)-\coind(h_\theta^1) =
\coSSS_{QPQ^{-1}}^\pm(\theta). 
\end{align*} 
\end{proof}

Consider two positive integers $d',d''$, and set $d:=d'+d''$. We identify $\R^{2d'}$ with the symplectic subspace $\R^{2d'}\times\{0\}\subset\R^{2d}$, and $\R^{2d''}$ with the symplectic subspace $\{0\}\times\R^{2d''}\subset\R^{2d}$. Given two symplectic matrices $P'\in\Sp(2d')$ and $P''\in\Sp(2d'')$, their direct sum is the symplectic matrix $P=P'\oplus P''\in\Sp(2d)$ given by $P(\zz',\zz'')=(P'\zz',P''\zz'')$. The next lemma shows that the splitting and cosplitting numbers behave naturally with respect to the direct sum operation.

\begin{lem}\label{l:direct_sum_splitting_numbers}
For all $P'\in\Sp(2d')$ and $P''\in\Sp(2d'')$, we have 
\begin{align*}
\SSS_{P'\oplus P''}^\pm(\theta) & =\SSS_{P'}^\pm(\theta)+\SSS_{P''}^\pm(\theta),\\ 
\coSSS_{P'\oplus P''}^\pm(\theta) & =\coSSS_{P'}^\pm(\theta)+\coSSS_{P''}^\pm(\theta).
\end{align*}
\end{lem}

\begin{proof}
For an integer $k>0$ large enough, we can find quadratic generating families $F':\R^{2d'k}\to\R$ and $F'':\R^{2d''k}\to\R$ for the matrices $P'$ and $P''$ (seen as Hamiltonian diffeomorphisms of $\R^{2d'}$ qnd $\R^{2d''}$ respectively). For each $\theta\in S^1$, we denote by $H_\theta'$ and $H_\theta''$ the $\theta$-Hessians of $F'$ and $F''$ at the origin, and by $h_\theta'$ and $h_\theta''$ the associated Hermitian bilinear forms. The function $F:\R^{2dk}\to\R$ given by $F(\zz',\zz'')=F'(\zz')+F''(\zz'')$ is a quadratic generating function for the matrix $P'\oplus P''$. Its $\theta$-Hessian at the origin is $H_\theta=H_\theta'\oplus H_\theta''$. In particular, index and coindex of the associated Hermitian  form $h_\theta$ satisfy
\begin{align*}
\ind(h_\theta) & =\ind(h_\theta')+\ind(h_\theta''),
\\
\coind(h_\theta) & =\coind(h_\theta')+\coind(h_\theta'').
\end{align*}
This implies the lemma.
\end{proof}

The following statement is the last ingredient that we need in order to prove the generalization of Theorem~\ref{t:iteration_inequality_nondegenerate}.

\begin{lem}\label{l:bound_splitting_numbers}
For all $P\in\Sp(2d)$ and $\theta\in S^1$, we have 
\begin{gather*}
0\leq\SSS_P^\pm(\theta)\leq\min\{\dim_\C\ker(P-\theta I),d\},\\
0\leq\coSSS_P^\pm(\theta)\leq\min\{\dim_\C\ker(P-\theta I),d\}.
\end{gather*}
\end{lem}

\begin{proof}
Notice that
\begin{align}\label{e:easy_bound_nullity}
\dim_\C\ker(P-\theta I)\leq d,\qquad\forall\theta\in S^1\setminus\{1,-1\}.
\end{align}
Indeed $\ker(P-\theta I)$ and $\ker(P-\overline\theta I)$ are vector subspaces of the same dimension (one is the complex conjugate of the other), and they have trivial intersection since $\theta\neq\overline\theta$. This, together with~\eqref{e:relation_splitting_cosplitting}, implies the bound of the lemma for  $\theta\not\in\{1,-1\}$.

The inequality~\eqref{e:easy_bound_nullity} does not hold for $\theta=\pm1$ (consider, for instance, the counterexample given by $P=I$ and $\theta=1$). The remaining bounds on the splitting numbers will be proved by equations~\eqref{e:ind_h_theta_restricted} and~\eqref{e:nul_h_theta_restricted}, which imply
\begin{align*}
\ind(h_\theta) = \ind(h_\theta|_{\VV\times\VV}) + \ind(h_\theta|_{\VV^{h_\theta}\times\VV^{h_\theta}}) + \nul(h_\theta|_{\VV^{h_\theta}\times\VV^{h_\theta}}) - \nul(h_\theta).
\end{align*}
We already remarked that the restricted form $h_\theta|_{\VV\times\VV}$ is independent of $\theta$, and therefore so is the summand $\ind(h_\theta|_{\VV\times\VV})$ in the above equation. Clearly 
\[\ind(h_\theta) - \ind(h_\theta|_{\VV\times\VV})\geq0.\]
Moreover 
\begin{align*}
\ind(h_\theta|_{\VV^{h_\theta}\times\VV^{h_\theta}}) + \nul(h_\theta|_{\VV^{h_\theta}\times\VV^{h_\theta}}) 
& \leq
\dim_\C \VV^{h_\theta}\\
& \leq
d + \dim_\C\ker(P-\theta I)\\
& =d + \nul(h_\theta).
\end{align*}
Therefore
\begin{align*}
S_P^\pm(\theta)
& =
\ind(h_{\theta^\pm}) - \ind(h_{\theta})\\
& \leq
\ind(h_{\theta^\pm}|_{\VV^{h_{\theta^\pm}}\times\VV^{h_{\theta^\pm}}}) + \nul(h_{\theta^\pm}|_{\VV^{h_{\theta^\pm}}\times\VV^{h_{\theta^\pm}}}) - \nul(h_{\theta^\pm})\\
& \leq d.
\end{align*}
This completes the proof of the bound for the splitting numbers. The one for the cosplitting numbers is proved by the same argument, with indices replaced by coindices.
\end{proof}

\subsection{The iteration inequality}

We can finally state and prove the general iteration inequality for the Morse index of generating families. We will adopt the notation of Section~\ref{s:Bott_indices}, so that $F_p:\R^{2dkp}\to\R$ denotes the generating function of the iterated Hamiltonian diffeomorphism $\phi^p\in\Ham(\R^{2d})$.

\begin{thm}[Iteration inequalities]
\label{t:iteration_inequalities}
Let $\zz=(z_0,...,z_{k-1})$ be a critical point of the generating function $F_1$, and let $p\in\N$. Then
\begin{equation}\label{e:iteration_inequalities}
\begin{split}
p\,\avind(\zz) - d & \leq \ind(\zz^p),\\ 
\ind(\zz^p)+\nul(\zz^p) & \leq p\,\avind(\zz) + d.
\end{split}
\end{equation}
If at least one of the above inequalities is an equality, then $\sigma(\diff\phi(z_0))=\{1\}$ and $\nul(\zz^p)\geq d$. Both inequalities are equalities if and only if $\diff\phi(z_0)^p=I$.
\end{thm}

\begin{rem}
Since $\ind(\zz^p)+\coind(\zz^p)+\nul(\zz^p)=2dkp$, the iteration inequalities~\eqref{e:iteration_inequalities} can be rewritten for the Morse coindex as
\begin{align*}
p\,\avcoind(\zz) - d & \leq \coind(\zz^p),\\ 
\tag*{\qed}
\coind(\zz^p)+\nul(\zz^p) & \leq p\,\avcoind(\zz) + d.
\end{align*}
\end{rem}

\begin{proof}[Proof of Theorem~\ref{t:iteration_inequalities}]
Let $\EE:=\ker(P-I)^{2d}\subset\R^{2d}$ be the generalized eigenspace of the eigenvalue 1 of the symplectic matrix $P:=\diff\phi(z_0)$. This vector subspace is symplectic (by Lemma~\ref{l:EE_1_is_symplectic}) and clearly invariant by $P$. Let $\EE^\omega$ be its symplectic orthogonal, that is
\begin{align*}
\EE^\omega = \big\{ Z\in\R^{2d}\ \big|\ \omega(Z,\cdot)|_{\EE}=0 \big\}.
\end{align*}
The space $\EE^\omega$ is also invariant by $P$. Indeed, if $Z\in\EE^\omega$, we have
\begin{align*}
\omega(PZ,Z')=\omega(Z,\underbrace{P^{-1}Z'}_{\in\EE})=0,\qquad \forall Z'\in\EE.
\end{align*}
Hence, by decomposing $\R^{2d}$ as the symplectic direct sum $\EE\oplus\EE^\omega$, the matrix $P$ takes the form $P'\oplus P''$, where $P'=P|_{\EE}$ and $P''=P|_{\EE^\omega}$. Notice that $P'$ is a unipotent matrix, i.e.\ $\sigma(P')=\{1\}$, while $\sigma(P'')$ does not contain $1$. Therefore, by Lemmata~\ref{l:direct_sum_splitting_numbers} and~\ref{l:bound_splitting_numbers}, we have
\begin{align}
\label{e:bound_splitting_numbers_inside_proof1}
\SSS_P^\pm(1) & =\SSS_{P'}^{\pm}(1)\leq \tfrac12 \dim_\C\EE,\\
\label{e:bound_splitting_numbers_inside_proof2}
\SSS_P^\pm(\theta) & =\SSS_{P''}^{\pm}(\theta)\leq \dim_\C(P''-\theta I),\qquad\forall\theta\in S^1\setminus\{1\}.
\end{align}
Analogously
\begin{align*}
\coSSS_P^\pm(1) & =\coSSS_{P'}^{\pm}(1)\leq \tfrac12 \dim_\C\EE,\\
\coSSS_P^\pm(\theta) & =\coSSS_{P''}^{\pm}(\theta)\leq \dim_\C(P''-\theta I),\qquad\forall\theta\in S^1\setminus\{1\}.
\end{align*}

We now proceed in a similar fashion as in Theorem~\ref{t:iteration_inequality_nondegenerate} (the argument will indeed reduce to that of Theorem~\ref{t:iteration_inequality_nondegenerate} if $\dim_\C\EE=0$). For all $\mu\in S^1$ with $\mathrm{Im}(\mu)>0$, we denote by $\sigma_\mu$ the set of eigenvalues of $P''$ on the unit circle $S^1$ with argument in the open interval $(0,\arg(\mu))$, and we set
\begin{align*}
f(\mu)&:= \sum_{\theta\in \sigma_\mu} \big(  \SSS_{P''}^+(\theta)-\SSS_{P''}^-(\theta) \big),\\
g(\mu)&:= \sum_{\theta\in \sigma_\mu} \big(  \coSSS_{P''}^+(\theta)-\coSSS_{P''}^-(\theta) \big).
\end{align*}
These functions are piecewise constant, with possible jumps only at the eigenvalues of $P''$. By Lemma~\ref{l:properties_ind_theta}(iii) and~\eqref{e:bound_splitting_numbers_inside_proof1}, if $\mu$ is not an eigenvalue of $P''$, we have
\begin{equation}\label{e:key_bound_for_index}
\begin{split}
\ind_{\mu}(\zz)-\ind(\zz)
&=\SSS_{P}^+(1)+ \sum_{\theta\in \sigma_\mu} \big(  \SSS_{P}^+(\theta)-\SSS_{P}^-(\theta) \big)\\
&= \SSS_{P'}^+(1)+ f(\mu)\\
&\leq \frac12 \dim_\C\EE + f(\mu).
\end{split} 
\end{equation}
By~\eqref{e:bound_splitting_numbers_inside_proof2} we have 
\begin{align*}
f(\mu)
 & \leq
 \sum_{\theta\in \sigma_{\mu}}  \dim_\C(P''-\theta I) \\
 & \leq
 \frac12\sum_{\theta\in S^1}  \dim_\C(P''-\theta I) \\
 & \leq
\frac12 \dim_\C\EE^\omega.
\end{align*}
Analogously, we have
\begin{align}
\label{e:bound_for_cosplitting_inside_proof}
\coind_{\mu}(\zz)-\coind(\zz)&\leq \frac12 \dim_\C\EE + g(\mu),\\
\nonumber
g(\mu) &\leq \frac12 \dim_\C\EE^\omega.
\end{align}
Notice that $\coind(\zz)=2dk-\ind(\zz)-\nul(\zz)$ and, since $\mu$ is not an eigenvalue of $P$, $\coind_\mu(\zz)=2dk-\ind_\mu(\zz)$. Therefore, the inequality~\eqref{e:bound_for_cosplitting_inside_proof} can be rewritten as
\begin{align}\label{e:key_bound_for_index_2}
\ind(\zz)+\nul(\zz)-\ind_\mu(\zz) \leq \frac12 \dim_\C\EE + g(\mu).
\end{align}
Let $\delta>0$ be such that there is no eigenvalue of $P''$ on the unit circle with argument in $[0,\delta]$. In particular, the functions $t\mapsto f(e^{it})$ and $t\mapsto g(e^{it})$ vanish on the interval $[0,\delta]$. By integrating~\eqref{e:key_bound_for_index} in $\mu$ on the upper semi-circle, we obtain
\begin{align*}
\avind(\zz)-\ind(\zz) &= \frac1\pi \int_0^{\pi} \big(\ind_{e^{it}}(\zz) - \ind(\zz)\big)\,\diff t
\\
 &\leq  \frac12 \dim_\C\EE + \frac1\pi \int_\delta^{\pi} f(e^{it})\,\diff t\\
 &\leq  \frac12 \left( \dim_\C\EE + \frac{\pi-\delta}\pi \dim_\C\EE^\omega \right).
\end{align*}
In particular 
\begin{align}\label{e:iteration_inequality_1}
\avind(\zz)-\ind(\zz)\leq\tfrac12(\dim_\C\EE+\dim_\C\EE^\omega)=d.
\end{align}
If this inequality is not strict, then $\dim_\C\EE^\omega=0$, that is, $\sigma(P)=\{1\}$. Moreover, in this case we have $\ind_\theta(\zz)=\avind(\zz)$ for all $\theta\neq 1$, so that $\SSS_P^\pm(1)=\avind(\zz)-\ind(\zz)=d$, and by~\eqref{e:relation_splitting_cosplitting} we conclude that $\nul(\zz)\geq \SSS_P^\pm(1)=d$.

If we now integrate~\eqref{e:key_bound_for_index_2}, we obtain
\begin{align*}
\ind(\zz)+\nul(\zz)-\avind(\zz) & = \frac1\pi \int_0^\pi \big(\ind(\zz)+\nul(\zz)-\ind_{e^{it}}(\zz) \big)\diff t\\
& \leq \frac12 \dim_\C\EE + \frac1\pi \int_\delta^\pi g(e^{it})\,\diff t\\
 &\leq  \frac12 \left( \dim_\C\EE + \frac{\pi-\delta}\pi \dim_\C\EE^\omega \right).
\end{align*}
Therefore 
\begin{align}\label{e:iteration_inequality_2}
\ind(\zz)+\nul(\zz)-\avind(\zz)\leq\tfrac12(\dim_\C\EE+\dim_\C\EE^\omega)=d.
\end{align}
As before, if this inequality is not strict, then $\dim_\C\EE^\omega=0$, that is, $\sigma(P)=\{1\}$. Moreover, in this case we have $\coind_\theta(\zz)=\avcoind(\zz)$ for all $\theta\neq 1$, so that 
\[
\coSSS_P^\pm(1)
=
\avcoind(\zz)-\coind(\zz)
=
\ind(\zz)+\nul(\zz)-\avind(\zz)
=
d,
\] 
and by~\eqref{e:relation_splitting_cosplitting} we conclude that $\nul(\zz)\geq \coSSS_P^\pm(1)=d$.

Both inequalities in~\eqref{e:iteration_inequality_1} and~\eqref{e:iteration_inequality_2} are simultaneously equalities if and only if $\nul(\zz)=2d$, that is, if and only if $P$ is the identity. This completes the proof of the theorem for period $p=1$. The case of an arbitrary period $p\in\N$ readily follows by recalling that $\avind(\zz^p)=p\,\avind(\zz)$.
\end{proof}

\subsection{Computation of splitting numbers}\label{s:computation_splitting_numbers}

We close this section by providing a recipe for computing the splitting numbers of a symplectic matrix $P\in\Sp(2d)$. We consider a quadratic generating family $F:\R^{2dk}\to\R$ for the linear Hamiltonian diffeomorphism $\phi(z)=Pz$. We denote by $H_\theta$ the $\theta$-Hessian of $F$, and by $h_\theta:\C^{2dk}\times\C^{2dk}\to\C$ the associated Hermitian bilinear form, so that in particular
\begin{align*}
F(\ZZ)=\tfrac12\langle H_1 \ZZ,\ZZ\rangle=\tfrac12 h_1(\ZZ,\ZZ),\qquad\forall\ZZ\in\R^{2dk}.
\end{align*}
In Section~\ref{s:splitting_numbers}, we studied the inertia of the restriction of $h_\theta$ to the vector subspace $\VV$ and to its $h_\theta$-orthogonal in order to show that the splitting and cosplitting numbers depend only on the considered symplectic matrix $P$. The choice of the vector space $\VV$ was suitable in order to establish the bounds of Lemma~\ref{l:bound_splitting_numbers}, but is not convenient for the numeric computation of the splitting and cosplitting numbers. For this purpose, we rather choose the vector space
\begin{align*}
\WW:=\big\{\ZZ=(Z_0,...,Z_{k-1})\in\C^{2dk}\ \big|\ Z_0=0\big\}.
\end{align*}
As in the case of $\VV$, the restriction of the Hermitian form $h_\theta$ to $\WW$ is independent of the parameter $\theta$, since for all $\ZZ,\ZZ'\in\C^{2dk}$ we have
\begin{align*}
h_\theta(\ZZ,\ZZ')
= & \,
\langle \overline\theta\, Y_{k-1} - Y_{0} + A_{k-1} X_0 + \overline\theta\, B_{k-1}^T Y_{k-1},X_0'\rangle \\
 & + \langle \theta\, X_{0} - X_{k-1} + \theta\,B_{k-1} X_{0} + C_{k-1} Y_{k-1},Y_{k-1}'\rangle \\
 & + \sum_{j=1}^{k-1} \langle Y_{j-1} - Y_{j} + A_{j-1} X_j + B_{j-1}^T Y_{j-1} , X_j'\rangle \\
 & + \sum_{j=0}^{k-2} \langle X_{j+1} - X_{j} + B_{j} X_{j+1} + C_{j} Y_{j},Y_j'\rangle\\
= & \,
 \langle - X_{k-1} + C_{k-1} Y_{k-1},Y_{k-1}'\rangle 
 + \langle -Y_1 + A_0 X_1, X_1' \rangle\\
 & + \sum_{j=2}^{k-1} \langle Y_{j-1} - Y_{j} + A_{j-1} X_j + B_{j-1}^T Y_{j-1} , X_j'\rangle \\
 & + \sum_{j=1}^{k-2} \langle X_{j+1} - X_{j} + B_{j} X_{j+1} + C_{j} Y_{j},Y_j'\rangle.
\end{align*}
In particular the functions $\theta\mapsto\ind(h_\theta|_{\WW\times\WW})$ and $\theta\mapsto\coind(h_\theta|_{\WW\times\WW})$ are independent of $\theta$. We recall that the kernel of $H_\theta$ is the space of vectors $\ZZ=(Z_0,...,Z_{k-1})\in\C^{2dk}$ such that $\phi_j Z_j=Z_{j+1}$ for all $j=0,...,k-2$, and $\phi_{k-1} Z_{k-1}=\theta\,Z_0$. Therefore
\begin{align*}
\WW\cap\ker H_\theta=\{0\}.
\end{align*}
The orthogonal vector space $\WW^{h_\theta}$ is given by the solutions $\ZZ\in\C^{2dk}$ of the following linear system
\begin{align*}
\theta\, X_{0} -  X_{k-1} + \theta\,B_{k-1} X_{0} + C_{k-1} Y_{k-1} & =0, \\
Y_{j-1} -  Y_{j} + A_{j-1} X_j + B_{j-1}^T Y_{j-1}  &=0,\qquad \forall j=1,...,k-1, \\
X_{j+1} - X_{j} + B_{j} X_{j+1} + C_{j} Y_{j} & =0,\qquad \forall j=1,...,k-2.
\end{align*}
Namely,
\begin{align*}
\WW^{h_\theta}=
\left\{
(Z_0,...,Z_{k-1})\in \C^{2dk}
\, \left|\ 
  \begin{array}{@{}l@{}}
    \phi_{0} (\tilde X_0,Y_0) =Z_1\mbox{ for some }\tilde X_0\in\C^d \vspace{5pt} \\
    \phi_j Z_j =Z_{j+1}\quad\forall j=1,...,k-2 \vspace{5pt} \\ 
    \phi_{k-1} Z_{k-1} =(\theta X_0,\tilde Y_k)\mbox{ for some }\tilde Y_k\in\C^d 
  \end{array}
\right.\right\}.
\end{align*}
We denote by $\Psi_\theta:\C^{2d}\to\WW^{h_\theta}$ the isomorphism given by $\Psi_\theta^{-1} \ZZ=(\tilde X_0,Y_0)$. Notice that $P\circ\Psi_\theta^{-1}(\ZZ)=(\theta X_0,\tilde Y_k)$.

The intersection
\begin{align*}
\WW\cap\WW^{h_\theta}=
\left\{
(0,Z_1,...,Z_{k-1})\in \C^{2dk}
\, \left|\ 
  \begin{array}{@{}l@{}}
    \phi_{0} (\tilde X_0,0) =Z_1\mbox{ for some }\tilde X_0\in\C^d \vspace{5pt} \\
    \phi_j Z_j =Z_{j+1}\quad\forall j=1,...,k-2 \vspace{5pt} \\ 
    \phi_{k-1} Z_{k-1} =(0,\tilde Y_k)\mbox{ for some }\tilde Y_k\in\C^d 
  \end{array}
\right.\right\}
\end{align*}
is independent of $\theta$. Proposition~\ref{p:index_restricted_form} gives
\begin{align*}
\ind(h_\theta)  =\ & \ind(h_\theta|_{\WW\times\WW}) + \ind(h_\theta|_{\WW^{h_\theta}\times\WW^{h_\theta}}) + \dim_\C(\WW\cap\WW^{h_\theta}),
\\
\coind(h_\theta)  =\ & \coind(h_\theta|_{\WW\times\WW}) + \coind(h_\theta|_{\WW^{h_\theta}\times\WW^{h_\theta}})  + \dim_\C(\WW\cap\WW^{h_\theta}).
\end{align*} 
Only the second summand in the right-hand sides of these two equations depends on $\theta$. Therefore, the splitting and cosplitting numbers are given by
\begin{align*}
\SSS_P^\pm(\theta) & =\ind(h_{\theta^\pm}|_{\WW^{h_{\theta^\pm}}\times\WW^{h_{\theta^\pm}}}) - \ind(h_\theta|_{\WW^{h_\theta}\times\WW^{h_\theta}}),\\
\coSSS_P^\pm(\theta) & =\coind(h_{\theta^\pm}|_{\WW^{h_{\theta^\pm}}\times\WW^{h_{\theta^\pm}}}) - \coind(h_\theta|_{\WW^{h_\theta}\times\WW^{h_\theta}}),
\end{align*}
Let us compute the restriction of the Hermitian form $h_\theta$ to $\WW^{h_\theta}$. For all pair of vectors $\ZZ,\ZZ'\in\WW^{h_\theta}$, we have
\begin{align*}
h_\theta(\ZZ,\ZZ')
= & \,
\langle \overline\theta\, Y_{k-1} - Y_{0} + A_{k-1} X_0 + \overline\theta\, B_{k-1}^T Y_{k-1},X_0'\rangle \\
 & +  \langle X_{1} - X_{0} + B_{0} X_{1} + C_{0} Y_{0},Y_0'\rangle\\
= & \,
\langle \overline\theta\, Y_{k-1} - Y_{0} + \overline\theta\,\underbrace{(A_{k-1} \theta\,X_0 +  B_{k-1}^T Y_{k-1})}_{\tilde Y_k - Y_{k-1}},X_0'\rangle \\
 & +  \langle X_{1} - X_{0} + \underbrace{B_{0} X_{1} + C_{0} Y_{0}}_{\tilde X_0 - X_1},Y_0'\rangle\\
= & \,
\langle \tilde Y_{k} - \theta\,Y_{0},\theta\, X_0'\rangle 
 +  \langle \tilde X_0 - \overline\theta\theta\,X_{0}, Y_0'\rangle,
\end{align*}
where $\tilde X_0$ and $\tilde Y_0$ depends on $\ZZ$ as in the above characterization of $\WW^{h_\theta}$.  Let us choose the more convenient coordinates given by the isomorphism $\Psi_\theta$. Namely, we consider the Hermitian form $g_\theta:\C^{2d}\times\C^{2d}\to\C$ given by 
\[g_\theta(Z,Z'):=h_{\theta}(\Psi_\theta Z,\Psi_\theta Z').\] 
If we write $(\tilde X,\tilde Y):=P(X,Y)$ and $(\tilde X',\tilde Y'):=P(X',Y')$, the Hermitian form $g_\theta$ can be written as
\begin{align}\label{e:Hermitian_form_for_splitting_numbers}
g_\theta((X,Y),(X',Y'))=
\langle \tilde Y - \theta\,Y,\tilde X' \rangle 
 +  \langle X - \overline\theta\, \tilde X , Y'\rangle
\end{align}
The splitting and cosplitting numbers can be conveniently computed as
\begin{align*}
\SSS_P^\pm(\theta) & =\ind(g_{\theta^\pm}) - \ind(g_\theta),\\
\coSSS_P^\pm(\theta) & =\coind(g_{\theta^\pm}) - \coind(g_\theta).
\end{align*}

\begin{exm}[Splitting numbers of a shear]\label{ex:splitting_numbers_shear}
For $r\in\R$, consider the unipotent symplectic matrix 
\begin{align*}
P
=
\left(
  \begin{array}{cc}
    1 & r \\ 
    0 & 1 \\ 
  \end{array}
\right).
\end{align*}
For each $\theta\in S^1$, the associated Hermitian form $g_\theta$ is given by
\begin{align*}
g_\theta(Z,Z')
&=
\langle (1-\theta)Y, X'+rY' \rangle 
 +  \langle (1-\overline\theta) X-\overline\theta r\,Y, Y'\rangle \\
&=(1-\theta) \langle Y,X'\rangle + (1-\overline\theta) \langle X,Y'\rangle
+ r(1-2\,\mathrm{Re}(\theta)) \langle Y,Y'\rangle
\end{align*}
The Hermitian matrix associated to $g_\theta$ is given by 
\begin{align*}
 \left(
  \begin{array}{cc}
    0 & 1-\theta \\ 
    1-\overline\theta & r(1-2\,\mathrm{Re}(\theta)) \\ 
  \end{array}
\right),
\end{align*}
whose eigenvalues $\lambda_1,\lambda_2\in\R$ satisfy $\lambda_1\lambda_2=-|1-\theta|^2$ and $\lambda_1+\lambda_2=r(1-2\,\mathrm{Re}(\theta))$. Therefore
\begin{align*}
\ind(g_\theta)&=\left\{
  \begin{array}{lll}
    1 &  & \mbox{if $\theta\neq1$ or $r>0$,} \\ 
    0 &  & \mbox{if $\theta=1$ and $r\leq0$,} \\ 
  \end{array}
\right. \vspace{10pt}\\
\coind(g_\theta)&=\left\{
  \begin{array}{lll}
    1 &  & \mbox{if $\theta\neq1$ or $r<0$,} \\ 
    0 &  & \mbox{if $\theta=1$ and $r\geq 0$,} \\ 
  \end{array}
\right.
\end{align*}
which implies
\begin{align*}
\SSS_P^\pm(1)&=\left\{
  \begin{array}{lll}
    1 &  & \mbox{if $r\leq0$,} \\ 
    0 &  & \mbox{if $r>0$,} \\ 
  \end{array}
\right. \vspace{10pt}\\
\coSSS_P^\pm(1)&=\left\{
  \begin{array}{lll}
    1 &  & \mbox{if $r\geq0$,} \\ 
    0 &  & \mbox{if $r<0$}. \\ 
  \end{array}
\right. 
\tag*{\qed}
\end{align*}
\end{exm}

\begin{exm}[Splitting numbers of a $\pi/2$-rotation]
Consider now the symplectic matrix of the standard complex structure of $(\R^{2d},\omega)$, which is
\begin{align*}
J
=
\left(
  \begin{array}{cc}
    0 & -I \\ 
    I & 0 \\ 
  \end{array}
\right).
\end{align*}
The eigenvalues of $J$ are $i$ and $-i$. The associated Hermitian forms $g_\theta$ are given by
\begin{align*}
g_\theta(Z,Z')
&=
\langle X-\theta Y, -Y' \rangle  
 +  \langle  X+\overline\theta Y, Y'\rangle,
\end{align*}
with associated Hermitian matrices
\begin{align*}
 \left(
  \begin{array}{cc}
    0 & 0 \\ 
    0 & (\theta+\overline{\theta})I \\ 
  \end{array}
\right),
\end{align*}
This readily implies
\begin{align*}
\ind(g_\theta)&=\left\{
  \begin{array}{lll}
    0 &  & \mbox{if $\mathrm{Re}(\theta)\geq0$,} \\ 
    d &  & \mbox{if $\mathrm{Re}(\theta)<0$,} \\ 
  \end{array}
\right. \vspace{10pt}\\
\coind(g_\theta)&=\left\{
  \begin{array}{lll}
    d &  & \mbox{if $\mathrm{Re}(\theta)>0$,} \\ 
    0 &  & \mbox{if $\mathrm{Re}(\theta)\leq0$,} \\ 
  \end{array}
\right.
\end{align*}
and therefore
\begin{align*}
\SSS_J^+(i)&=\SSS_J^-(-i)=\coSSS_J^-(i)=\coSSS_J^+(-i)=d,\\
\SSS_J^-(i)&=\SSS_J^+(-i)=\coSSS_J^+(i)=\coSSS_J^-(-i)=0.
\tag*{\qed}
\end{align*}
\end{exm}

\subsection{Bibliographical remarks}\label{ss:Bott_biblio_remarks}

The iteration theory for the Morse indices of periodic orbits was introduced in the setting of Tonelli Lagrangian systems by Bott  \cite{Bott:On_the_iteration_of_closed_geodesics_and_the_Sturm_intersection_theory},  who developed ideas introduced earlier by Hedlund~\cite{Hedlund:Poincare_s_rotation_number_and_Morse_s_type_number} and Morse-Pitcher~\cite{Morse_Pitcher:On_certain_invariants_of_closed_extremals}. The setting of this section is more general than Bott's one, as we will discuss in Section~\ref{s:Lagrangian}. A special Morse index theory for the Hamiltonian action functional was first studied by Conley and Zehnder in their papers \cite{Conley_Zehnder:Morse_type_index_theory_for_flows_and_periodic_solutions_for_Hamiltonian_equations, Conley_Zehnder:Subharmonic_solutions_and_Morse_theory}. As we already mentioned before Proposition~\ref{p:Morse_idx_always_large}, the critical points of the Hamiltonian action functional always have infinite Morse index and coindex. The index that Conley and Zehnder defined coincides with the Maslov index, which we will introduce in Section~\ref{s:Maslov_index}. Theorem~\ref{t:iteration_inequality_nondegenerate} is the translation, in the finite dimensional setting of Chaperon's generating families, of Conley-Zehnder's iteration inequality for the Maslov index of non-degenerate symplectic paths. This inequality is the crucial ingredient in the proof of one of Conley-Zehnder's famous theorems from~\cite{Conley_Zehnder:Subharmonic_solutions_and_Morse_theory} (see also the author's~\cite{Mazzucchelli:Symplectically_degenerate_maxima_via_generating_functions} for a proof using Chaperon's generating families, and Salamon-Zehnder's \cite{Salamon_Zehnder:Morse_theory_for_periodic_solutions_of_Hamiltonian_systems_and_the_Maslov_index} for a generalization to all closed symplectically aspherical manifolds): a generic Hamiltonian diffeomorphism on a standard symplectic $2d$-torus possesses  periodic points of arbitrarily large minimal period. The general iteration inequalities, or more precisely their translation in terms of the Maslov indices (Theorem~\ref{t:iteration_inequalities_Maslov}), are due to Liu and Long~\cite{Liu_Long:An_optimal_increasing_estimate_of_the_iterated_Maslov-type_indices, Liu_Long:Iteration_inequalities_of_the_Maslov-type_index_theory_with_applications}. One of their most remarkable application is to the non-generic version of Conley-Zehnder's Theorem, which was a long standing conjecture due to Conley and established by Hingston~\cite{Hingston:Subharmonic_solutions_of_Hamiltonian_equations_on_tori}: any Hamiltonian diffeomorphism on a standard symplectic $2d$-torus with finitely many fixed points possesses periodic points of arbitrarily large minimal period. Generalizations of Hingston's Theorem to larger and larger classes of closed symplectic manifolds were established by Ginzburg~\cite{Ginzburg:The_Conley_conjecture}, Ginzburg-G\"urel \cite{Ginzburg_Gurel:Local_Floer_homology_and_the_action_gap, Ginzburg_Gurel:Conley_Conjecture_for_Negative_Monotone_Symplectic_Manifolds} and Hein~\cite{Hein:The_Conley_conjecture_for_irrational_symplectic_manifolds}. We refer the reader to Long's monograph~\cite{Long:Index_theory_for_symplectic_paths_with_applications} for other applications of the iteration inequalities. Many proofs that we provided in this section, as well as the recipe for computing the splitting numbers of symplectic matrices, were inspired by Ballmann-Thorbergsson-Ziller's~\cite{Ballmann_Thorbergsson_Ziller:Closed_geodesics_on_positively_curved_manifolds}.

\section{The Maslov index}\label{s:Maslov_index}

\subsection{Behavior of the inertia indices under stabilization}

Consider a symplectic matrix $P\in\Sp(2d)$. Choose a factorization 
\begin{align}\label{e:factorization_P}
P=P_{k-1}\circ...\circ P_0 
\end{align}
such that each $P_j$ is sufficiently close to the identity in $\Sp(2d)$, and therefore it is described by a quadratic generating function $f_j:\R^{2d}\to\R$. As before, we write this function as
\begin{align*}
f_j(X_{j+1},Y_j)=
\tfrac12\langle A_j X_{j+1},X_{j+1}\rangle
+
\langle B_j X_{j+1},Y_j\rangle
+
\tfrac12\langle C_j Y_j,Y_j\rangle,
\end{align*}
where $A_j$, $B_j$, and $C_j$ are (small) $dk\times dk$ real matrices, $A_j$ and $C_j$ being symmetric. The factorization~\eqref{e:factorization_P} singles out a path in the symplectic group joining the identity to $P$. Indeed, for all $t\in[0,1]$, let $P_j^t$ be the symplectic matrix defined by the generating function $t\,f_j$, i.e.
\begin{align*}
P_j^t Z_j=Z_{j+1}\qquad\mbox{if and only if}\qquad
\left\{
  \begin{array}{l}
    X_{j+1}-X_j=-t(B_j X_{j+1} + C_j Y_j), \\ 
    Y_{j+1}-Y_j=t(A_j X_{j+1}+B_j^T Y_j ).
  \end{array}
\right.
\end{align*}
Notice that $P_j^0=I$ and $P_j^1=P_j$. For all $t\in[0,1]$, we set
\begin{align*}
P^t=P_j^s\circ P_{j-1}\circ...\circ P_0,\qquad\mbox{where }j=\lfloor kt\rfloor,\ s=kt-j.
\end{align*}
The continuous path $t\mapsto P^t$ in the symplectic group $\Sp(2d)$ joins  $P^0=I$ and $P^1=P$. 

On the other hand, if we started with a continuous path $\Gamma:[0,1]\to\Sp(2d)$ such that $\Gamma(0)=I$ and $\Gamma(1)=P$, up to choosing $k$ large enough, for all $|t_1-t_2|\leq1/k$ the symplectic matrix $\Gamma(t_2)\Gamma(t_1)^{-1}$ becomes as close to the identity as we wish, and in particular close enough to being described by a quadratic generating function. If we now set 
\begin{align}\label{e:discretization_of_symplectic_path}
 P_j:=\Gamma(\tfrac{j+1}{k})\Gamma({\tfrac jk})^{-1},\qquad \forall j=0,...,k-1,
\end{align}
and denote by $t\mapsto P^t$ the symplectic path associated to the factorization~\eqref{e:factorization_P} as above, the paths $\Gamma$ and $t\mapsto P^t$ are homotopic (via a homotopy that fixes the endpoints). Indeed, their restrictions to any time interval of the form $[j/k,(j+1)/k]$ are homotopic with fixed endpoints.

Let $F:\R^{2dk}\to\R$ be the quadratic generating family associated to the  factorization~\eqref{e:factorization_P} of $P$, that is,
\begin{align}\label{e:quadratic_gf_associated_to_discretization}
F(\ZZ) 
= 
\tfrac12\langle H \ZZ,\ZZ\rangle
=
\sum_{j\in\Z_k} \Big( \langle Y_j,X_{j+1}-X_j\rangle + f_j(X_{j+1},Y_j) \Big).
\end{align}
We denote by $h(\ZZ,\ZZ')=\langle H \ZZ,\ZZ'\rangle$ the Hessian bilinear form associated to $F$. We recall that an image vector $\ZZ'=H(\ZZ)$ is defined by
\begin{align*}
X_j' & =Y_{j-1}-Y_{j} + A_{j-1} X_j + B_{j-1}^T Y_{j-1},\\
Y_j' & =X_{j+1}-X_{j} + B_{j} X_{j+1} + C_{j} Y_{j}. 
\end{align*}
From Lemma~\ref{l:properties_ind_theta}(ii), we know that $\nul(h)=\dim\ker(P-I)$. However, it is not hard to convince ourselves that the data of $P$ alone is not enough to determine the other inertia indices of $h$ (see for instance the example of the identity mentioned before Proposition~\ref{p:Morse_idx_always_large}). In this section we are going to show that the index $\ind(h)$ and the coindex $\coind(h)$ are completely determined by the number of factors $k$ in the factorization~\eqref{e:factorization_P} and by the homotopy class of the path $t\mapsto P^t$ in the symplectic group $\Sp(2d)$.

Let us begin by studying how the inertia indices change if we increase $k$ by adding trivial factors in~\eqref{e:factorization_P}. For some $l>k$, let us set $P_k=P_{k+1}=...=P_{l-1}:=I$, and consider the generating function $F':\R^{2dl}\to\R$ associated to the factorization 
\[P=P_{l-1}\circ...\circ P_k\circ P_{k-1}\circ...\circ P_0.\] We denote by $h'(\ZZ,\ZZ')=\langle H'\ZZ,\ZZ'\rangle$ the Hessian bilinear form associated to $F'$. The following lemma shows that $h'$ is essentially a stabilization of $h$.

\begin{lem}\label{l:stabilization}
The inertia indices of $h$ and $h'$ are related by
\begin{align*}
\nul(h')&=\nul(h),\\
\ind(h')&=\ind(h)+d(l-k),\\
\coind(h')&=\coind(h)+d(l-k).
\end{align*}
\end{lem}

\begin{proof}
We already know the claim about the nullities, so let us focus on the other two. Consider the vector space
\begin{align*}
\VV=\big\{\ZZ\in\R^{2dl}\ \big|\ Z_{k}=Z_{k+1}=...=Z_{l-1}=Z_0\big\}.
\end{align*}
Let $\pi:\R^{2dl}\to\R^{2dk}$ be the projection
\[\pi(X_0,Y_0,...,X_{l-1},Y_{l-1})=(X_k,Y_0,X_1,Y_1,...,X_{k-1},Y_{k-1}),\]
and  $\iota:\R^{2dk}\to\VV$  the isomorphism 
\[\iota(Z_0,...,Z_{k-1})=\iota(Z_0,...,Z_{k-1},Z_0,Z_0,...,Z_0).\] Notice that 
\begin{align}\label{e:restriction_to_VV}
h'(\ZZ,\iota(\ZZ'))=h(\pi(\ZZ),\ZZ'),\qquad\forall \ZZ\in\R^{2dl},\ZZ'\in\R^{2dk}.
\end{align}
Since the inverse of $\iota$ is given by $\pi|_\VV$, the restriction of $h'$ to $\VV$ coincides with $h$, in the sense that $h'(\iota(\cdot),\iota(\cdot))=h$. Therefore
\begin{align*}
\ind(h'|_{\VV\times\VV})=\ind(h),\qquad
\coind(h'|_{\VV\times\VV})=\coind(h).
\end{align*}
Now, we need to study the $h'$-orthogonal to $\VV$. By~\eqref{e:restriction_to_VV}, we infer that 
\[\VV^{h'}=\pi^{-1}(\ker(h)).\] 
Namely, $\VV^{h'}$ is the vector space of the solutions $\ZZ\in\R^{2dl}$ of the linear system
\begin{align*}
X_{1} - X_{k} + B_{0} X_{1} + C_{0} Y_{0} & =0,\\
 X_{j+1} - X_{j} + B_{j} X_{j+1} + C_{j} Y_{j} & = 0, &\forall j=1,...,k-1,\\
Y_{k-1} - Y_{0} + A_{k-1} X_{k} + B_{k-1}^T Y_{k-1} & =0,\\
 Y_{j-1} - Y_{j} + A_{j-1} X_j + B_{j-1}^T Y_{j-1} & =0, &\forall j=1,...,k-1.
\end{align*}
This means that
\begin{align*}
\VV^{h'}=\left\{(Z_0,...,Z_{l-1})\in\R^{2dl}\ 
\left|\ 
  \begin{array}{@{}l@{}}
    P_0(X_k,Y_0) =(X_1,Y_1) \vspace{5pt} \\ 
    P_j(X_j,Y_j) =(X_{j+1},Y_{j+1})\quad\forall j=1,...,k-2 \vspace{5pt}\\
    P_{k-1}(X_{k-1},Y_{k-1}) =(X_k,Y_0)   
  \end{array}
\right.\right\}.
\end{align*}
Notice that $\VV\cap\VV^{h'}=\ker H=\VV\cap\ker H$. By Proposition~\ref{p:index_restricted_form} we infer that
\begin{align*}
\ind(h') & = \ind(h) + \ind(h'|_{\VV^{h'}\times\VV^{h'}}),\\
\coind(h') & = \coind(h) + \coind(h'|_{\VV^{h'}\times\VV^{h'}}).
\end{align*}
In order to complete the proof, we only have to show that
\begin{align}\label{e:index_at_infinity}
\ind(h'|_{\VV^{h'}\times\VV^{h'}})=\coind(h'|_{\VV^{h'}\times\VV^{h'}})=d(l-k).
\end{align}
We prove this equality as follows. For all $\ZZ,\ZZ'\in\VV^{h'}$, we have
\begin{align*}
h'(\ZZ,\ZZ')
= &
\sum_{j=k}^{l-1}
\Big(\langle Y_j-Y_{j+1},X_{j+1}'\rangle + \langle X_{j+1}-X_j,Y_j'\rangle\Big)\\
& + \langle -Y_k + \underbrace{Y_{k-1} + A_{k-1}X_k + B_{k-1}^T Y_{k-1}}_{=Y_0},X_{k}' \rangle\\
& + \langle -X_0 + \underbrace{X_1 + B_0 X_1 + C_0 Y_0}_{=X_k},Y_0' \rangle.
\end{align*}
This expression readily implies that the symmetric bilinear form $h'|_{\VV^{h'}\times\VV^{h'}}$ is negative definite on the following vector subspace of $\VV^{h'}$
\begin{align*}
\EE^-
=
\left\{(X_0,0,...,0,Y_k,X_{k+1},Y_{k+1},...,X_{l-1},Y_{l-1})\in\R^{2dl}\ 
\left|\ 
  \begin{array}{@{}r@{}}
    X_{j+1}=Y_{j+1}-Y_j \vspace{5pt} \\ 
    \forall j=k,...,l-1  
  \end{array}
\right.\right\}.
\end{align*}
Analogously, $h'|_{\VV^{h'}\times\VV^{h'}}$ is positive definite on the following vector subspace of $\VV^{h'}$
\begin{align*}
\EE^+
=
\left\{(X_0,0,...,0,Y_k,X_{k+1},Y_{k+1},...,X_{l-1},Y_{l-1})\in\R^{2dl}\ 
\left|\ 
  \begin{array}{@{}r@{}}
    X_{j+1}=Y_j-Y_{j+1} \vspace{5pt} \\ 
    \forall j=k,...,l-1  
  \end{array}
\right.\right\}.
\end{align*}
Notice that $\dim\EE^-=\dim\EE^+=d(l-k)$, and obviously the intersection of  $\EE^-$ with $\EE^+$ is trivial. Notice further that the kernel of $h'|_{\VV^{h'}\times\VV^{h'}}$ is given by
\begin{align*}
\ker(h'|_{\VV^{h'}\times\VV^{h'}})=
\left\{(Z_0,...,Z_{l-1})\in\VV^{h'}\ 
\left|\ 
  \begin{array}{@{}l@{}}
    X_0=X_k=X_{k+1}=...=X_{l-1} \vspace{5pt} \\ 
    Y_0=Y_k=Y_{k+1}=...=Y_{l-1}
  \end{array}
\right.\right\},
\end{align*}
whose dimension is $\nul(h'|_{\VV^{h'}\times\VV^{h'}})=\dim\ker(P-I)$. Finally, we remark that there is an isomorphism 
\[\Psi:\VV^{h'}\to\ker(P-I)\times\R^{2d(l-k)}\] given by
\begin{align*}
\Psi(Z_0,...,Z_{l-1})=
\big((X_k,Y_0),(X_0,Y_k,X_{k+1},Y_{k+1},X_{k+2},Y_{k+2},...,X_{l-1},Y_{l-1})\big).
\end{align*}
In particular 
\begin{align*}
\dim\VV^{h'}=\dim\ker(P-I)+2d(l-k)=\nul(h'|_{\VV^{h'}\times\VV^{h'}}) + \dim(\EE^-) + \dim(\EE^+).
\end{align*}
Therefore $\EE^-$ and $\EE^+$ are maximal vector subspaces of $\VV^{h'}$ where the bilinear form $h'|_{\VV^{h'}\times\VV^{h'}}$ is negative definite and positive definite respectively. This implies our claim in~\eqref{e:index_at_infinity}.
\end{proof}

\subsection{Morse and Maslov indices}
We now have all the ingredients to introduce the main character of this section: the Maslov index. For $n=0,...,2d$, we introduce the spaces of symplectic paths
\begin{align*}
 \PP_n(2d):=\big\{ \Gamma:[0,1]\to\Sp(2d)\  \big|\ \Gamma(0)=I,\ \dim\ker(\Gamma(1)-I)=n  \big\}
\end{align*}
endowed with the $C^0$-topology. This gives a partition of the full space of symplectic paths
\begin{align*}
 \PP(2d):=\bigcup_{n=0}^{2d} \PP_n(2d).
\end{align*}
Given $\Gamma\in\PP$ with $\Gamma(1)=:P$, we choose a parameter $k$ large enough and we consider the symplectic matrices~\eqref{e:discretization_of_symplectic_path}, which give the factorization~\eqref{e:factorization_P} and the associated quadratic generating function with Hessian bilinear form $h:\R^{2dk}\times\R^{2dk}\to\R$. We define the \textbf{Maslov index} of $\Gamma$ as
\begin{align*}
\mas(\Gamma):=\ind(h) - dk\in\Z.
\end{align*}
Analogously, we defined the \textbf{Maslov coindex} of $\Gamma$ as
\begin{align*}
\comas(\Gamma):=\coind(h) - dk\in\Z.
\end{align*}
Since the inertia indices are related by $\ind(h)+\coind(h)+\nul(h)=2dk$, we have
\begin{align}\label{e:relation_Maslov_coMaslov}
\mas(\Gamma) + \comas(\Gamma) + \dim\ker(\Gamma(1)-I) = 0.
\end{align}
In particular the Maslov index is equal to minus the Maslov coindex on the subspace $\PP_0(2d)$.

\begin{war}\label{w:convention_Maslov_index}
Many authors in symplectic topology call Maslov index what we call Maslov coindex. This different convention amounts to changing the sign of the generating families. Example~\ref{ex:rotation} below can be useful to recognize the sign convention adopted in a paper. \hfill\qed
\end{war}

The next Theorem implies that these are good definitions.
\begin{thm}\label{t:Maslov_index_homotopy_invariance}
The Maslov index is a well defined function \[\mas:\PP(2d)\to\Z,\] i.e.\ $\mas(\Gamma)$ is independent of the chosen parameter $k$. Moreover, it is a lower semi-continuous function, and it is locally constant on every subspace $\PP_n(2d)$. The same properties hold for the Maslov coindex.
\end{thm}

\begin{proof}
In order to show that the Maslov index and coindex are well defined, we only have to prove that we obtain the same indices if we replace $k$ by a larger parameter $l$ in the setting above. We proceed as follows. We define a homotopy $\Gamma_s:[0,1]\to\Sp(2d)$, for $s\in[0,1]$, such that $\Gamma_0=\Gamma$, each $\Gamma_s$ has the same endpoints as $\Gamma$, and $\Gamma_1$ runs along the whole $\Gamma$ in the time interval $[0,k/l]$, and stays constant at $\Gamma(1)$ in the remaining time interval $[k/l,1]$. This homotopy is defined by the formula
\begin{align*}
 \Gamma_s(t):=\Gamma\big(\min\big\{1,\tfrac{l}{l+s(k-l)}t\big\}\big).
\end{align*}
For each $s\in[0,1]$, we introduce the factorization 
\begin{align}\label{e:factorization_of_homotopy}
\Gamma(1)=P_{l-1,s}\circ P_{l-2,s}\circ...\circ P_{0,s}, 
\end{align}
where
\begin{align*}
 P_{j,s}:=\Gamma_s(\tfrac{j+1}{l})\Gamma_s(\tfrac{j}{l})^{-1},\qquad\forall j=0,...,l-1.
\end{align*}
Since the parameter $l$ is larger than $k$, each symplectic matrix $P_{j,s}$ is sufficiently close to the identity to be described by a quadratic generating function. We denote by $h_s:\R^{2dl}\times\R^{2dl}\to\R$ the Hessian bilinear form of the quadratic generating family associated to the factorization~\eqref{e:factorization_of_homotopy}. For $s=1$, equation~\eqref{e:factorization_of_homotopy} gives the factorization of $\Gamma(1)$ corresponding to the parameter $l$. Hence, all we need to do is to prove that
\begin{align*}
\ind(h_1) - dl & = \ind(h) - dk,\\
\coind(h_1) - dl & = \coind(h) - dk.
\end{align*}
Notice that $h_s$ depends continuously on $s\in[0,1]$. Moreover, its nullity is constant in $s$, since
\begin{align*}
\nul(h_s) = \dim\ker(\Gamma(1)-I),\qquad\forall s\in[0,1].
\end{align*}
This implies that the functions $s\mapsto\ind(h_s)$ and $s\mapsto\coind(h_s)$ are  constant in $s$ as well. For $s=0$, equation~\eqref{e:factorization_of_homotopy} gives the following factorization of $\Gamma(1)$
\begin{align*}
 \Gamma(1)= \underbrace{I\circ I\circ ...\circ I}_{\times l-k}\circ P_{k-1}\circ P_{k-2}\circ ...\circ P_0.
\end{align*}
By Lemma~\ref{l:stabilization}, we have
\begin{align*}
\ind(h_0)&=\ind(h)+d(l-k),\\ 
\coind(h_0)&=\coind(h)+d(l-k).
\end{align*}
Therefore
\begin{align*}
\ind(h_1) - dl & = \ind(h_0) - dl = \ind(h)+d(l-k)-dl = \ind(h)-dk,\\
\coind(h_1) - dl & = \coind(h_0) - dl = \coind(h)+d(l-k)-dl = \coind(h)-dk.
\end{align*}
This completes the proof that the Maslov index and coindex are well defined. Their lower semi-continuity follows immediately by the same property for the inertia index and coindex of symmetric bilinear forms.

Finally, let $s\mapsto\Gamma_s$, $s\in[0,1]$, be a path inside a space $\PP_n(2d)$, for some $n\in\N$. Notice that $\Gamma_s(0)=I$ and $\dim\ker(\Gamma_s(1)-I)=n$, but the path of symplectic matrices $s\mapsto\Gamma_s(1)$ does not have to be constant. For $k\in\N$ large enough, let us introduce the factorization 
\begin{align*}
\Gamma_s(1)=P_{k-1,s}\circ P_{k-2,s}\circ...\circ P_{0,s}, 
\end{align*}
where $P_{j,s}:=\Gamma_s(\tfrac{j+1}{l})\Gamma_s(\tfrac{j}{l})^{-1}$. We denote by $h_s:\R^{2dk}\times\R^{2dk}\to\R$ the Hessian bilinear form associated to this factorization of $\Gamma_s(1)$. As before, $h_s$ depends continuously on $s$, and its nullity is constantly equal to $n$. This implies that the functions $s\mapsto\ind(h_s)$ and $s\mapsto\coind(h_s)$ are constant. In particular $\mas(\Gamma_0)=\mas(\Gamma_1)$ and $\comas(\Gamma_0)=\comas(\Gamma_1)$.
\end{proof}

\subsection{Bott's iteration theory for the Maslov index}

By combining Sections~\ref{s:Bott} and~\ref{s:Maslov_index}, we obtain an iteration theory for the Maslov index and coindex. Consider a continuous path $\Gamma:[0,1]\to\Sp(2d)$ with $\Gamma(0)=I$. Fix a parameter $k$ large enough, and consider the factorization $\Gamma(1)=P_{k-1}\circ...\circ P_0$ whose factors are defined by~\eqref{e:discretization_of_symplectic_path}, and the associated quadratic generating family $F:\R^{2dk}\to\R$ given by~\eqref{e:quadratic_gf_associated_to_discretization}. For $\theta\in S^1$, let $H_\theta$ be the $\theta$-Hessian of $F$, and $h_\theta:\C^{2dk}\times\C^{2dk}\to\C$ the associated Hermitian bilinear form. We defined the \textbf{$\theta$-Maslov index} and \textbf{coindex} of $\Gamma$ as
\begin{align*}
\mas_\theta(\Gamma):=\ind(h_\theta)-dk,\qquad
\comas_\theta(\Gamma):=\coind(h_\theta)-dk,
\end{align*}
so that
\begin{align*}
\mas_\theta(\Gamma) + \comas_\theta(\Gamma) + \dim\ker(\Gamma(1)-\theta I) = 0.
\end{align*}
These indices are well defined independently of the sufficiently large parameter $k$ employed. Indeed, for $\theta=1$ these are the standard Maslov index and coindex, and the fact that they are independent of $k$ was already proved in the previous subsection. Moreover, Lemma~\ref{p:Bott_function_dependence} implies that the functions $\theta\mapsto\mas_\theta(\Gamma)-\mas_1(\Gamma)$ and $\theta\mapsto\comas_\theta(\Gamma)-\comas_1(\Gamma)$ are completely determined by the symplectic matrix $\Gamma(1)$.

Theorem~\ref{t:Maslov_index_homotopy_invariance} is generalized by the following.

\begin{thm}
The $\theta$-Maslov index $\mas_\theta:\PP(2d)\to\Z$ is a lower semi-continuous function and, for each $n\in\N$, is locally constant on the subspace 
\begin{align*}
\PP_{\theta,n}(2d) :=  
\big\{ \Gamma:[0,1]\to\Sp(2d)\  \big|\ \Gamma(0)=I,\ \dim\ker(\Gamma(1)-\theta I)=n  \big\}.
\end{align*}
The same properties hold for the Maslov coindex.
\end{thm}

\begin{proof}
The proof is entirely analogous to the one of Theorem~\ref{t:Maslov_index_homotopy_invariance}. Briefly, the lower semi-continuity of the $\theta$-Maslov index follows from the same property for the index of Hermitian bilinear forms. As for the other claim, consider a path $s\mapsto\Gamma_s$ inside a subspace $\PP_{\theta,n}(2d)$. For $k$ large enough, there exists a continuous family $h_{s,\theta}:\C^{2dk}\times\C^{2dk}\to\C$ of associated $\theta$-Hessian Hermitian bilinear forms. Since 
\[\nul(h_{s,\theta})=\dim\ker(\Gamma_s(1)-\theta I)=n,\] 
we readily have that the functions $s\mapsto\ind(h_{s,\theta})$ and $s\mapsto\coind(h_{s,\theta})$ are constant, and so are the functions $s\mapsto\mas(\Gamma_s)$ and $s\mapsto\comas(\Gamma_s)$.
\end{proof}

Let us provide the motivation for the introduction of such generalized Maslov indices. We define the \textbf{$p$-th iteration} of $\Gamma$ as the continuous path $\Gamma_{p}:[0,1]\to\Sp(2d)$ given by
\begin{align*}
\Gamma_p(\tfrac{j+t}{p})=\Gamma(t)\Gamma(1)^j,\qquad\forall j=0,...,p-1,\ t\in[0,1].
\end{align*}
This notion arises naturally in the context of periodic Hamiltonian systems. Indeed, assume that $H_t:\R^{2d}\to\R$ is a smooth non-autonomous Hamiltonian that is 1-periodic in time, i.e.\ $H_{t+1}=H_t$ for all $t\in\R$. If $H$ defines a global Hamiltonian flow $\phi_t$, this verifies $\phi_{t+1}=\phi_t\circ\phi_1$ for all $t\in\R$. If now $z$ is a fixed point of $\phi_1$, we can linearize the flow at $z$, thus obtaining the symplectic path $\Gamma:\R\to\Sp(2d)$ given by $\Gamma(t)=\diff\phi_t(z)$. The $p$-th iteration of the path $\Gamma|_{[0,1]}$ is the reparametrization of the path $\Gamma|_{[0,p]}$ given by $\Gamma_p(t)=\Gamma(pt)$, for $t\in[0,1]$.

The iterated path $\Gamma_p$ defines a factorization of $\Gamma_p(1)=\Gamma(1)^p$ which is precisely the $p$-th fold juxtaposition of the original factorization $P_{k-1}\circ...\circ P_0$ of $\Gamma(1)$. This puts ourselves in the setting of Section~\ref{s:Bott}. In particular, Bott's formulae of Lemma~\ref{l:Bott_formulae} can be stated for the Maslov index and coindex as
\begin{equation}
\label{e:Bott_formulae_Maslov}
\begin{split}
\mas_\theta(\Gamma_p)&=\sum_{\mu\in\sqrt[p]\theta} \mas_\mu(\Gamma),\\
\comas_\theta(\Gamma_p)&=\sum_{\mu\in\sqrt[p]\theta} \comas_\mu(\Gamma). 
\end{split}
\end{equation}
In particular, the Maslov index and coindex of any $p$-th iterate of $\Gamma$ are completely determined by the functions $\theta\mapsto\mas_\theta(\Gamma)$ and $\theta\mapsto\comas_\theta(\Gamma)$ respectively.

\begin{exm}\label{ex:rotation}
Let us compute the Maslov index and coindex of the symplectic path $\Gamma:[0,1]\to\Sp(2)$ given by rigid rotations from angle $0$ to some angle $\beta>0$, i.e.
\begin{align*}
\Gamma(t)
=
\left(
  \begin{array}{cc}
    \cos(t\beta) & -\sin(t\beta) \\ 
    \sin(t\beta) & \cos(t\beta) \\ 
  \end{array}
\right).
\end{align*}
Let $p\in\N$ be large enough so that, for all $t\in[0,1/p]$, the symplectic matrix $\Gamma(t)$ is described by a generating function. This is verified precisely when the angle $\alpha:=\beta/p$ lies in the interval $(0,\pi/2)$. We denote by $\Upsilon:[0,1]\to\Sp(2)$ the continuous path $\Upsilon(t):=\Gamma(pt)$, so that $\Gamma$ is the $p$-th iteration of $\Upsilon$. The matrix $\Upsilon(1)$ is described by the quadratic generating function $F(z)=\tfrac12 h(z,z)=\tfrac12\langle Hz,z\rangle$ whose Hessian matrix is
\begin{align*}
H
=
\left(
  \begin{array}{cc}
    \tan(\alpha) & \cos(\alpha)^{-1}-1 \\ 
    \cos(\alpha)^{-1}-1 & \tan(\alpha) 
  \end{array}
\right).
\end{align*}
The eigenvalues $\lambda_1,\lambda_2\in\R$ of this matrix satisfy $\lambda_1\lambda_2=(1-\cos(\alpha))\cos(\alpha)^{-2}$ and $\lambda_1+\lambda_2=2\tan(\alpha)$. Therefore
\begin{align*}
\mas(\Upsilon)&=\ind(h)-1=-1,\\
\comas(\Upsilon)&=\coind(h)-1=1.
\end{align*}
In order to compute the $\theta$-Maslov indices we can make the same computation with the $\theta$-Hessian of the generating function $F$, or equivalently apply the recipe from Section~\ref{s:computation_splitting_numbers}. Let us choose the second option. For all $\theta\in S^1$, the Hermitian bilinear form $g_\theta$ associated to the symplectic matrix $P=\Upsilon(1)$ as in~\eqref{e:Hermitian_form_for_splitting_numbers} is given by
\begin{align*}
g_\theta((X,Y),(X',Y'))
=\, &
\sin(\alpha)\cos(\alpha)\langle X,X'\rangle \\
&+
\cos(\alpha)(\cos(\alpha)-\theta)\langle Y,X' \rangle\\
&+ \cos(\alpha)(\cos(\alpha)-\overline\theta)\langle X,Y' \rangle \\
&+\sin(\alpha)(2\,\mathrm{Re}(\theta)-\cos(\alpha))\langle Y,Y'\rangle.
\end{align*}
The eigenvalues $\kappa_1,\kappa_2\in\R$ of the associated Hermitian matrix 
\begin{align*}
\left(
  \begin{array}{cc}
    \sin(\alpha)\cos(\alpha) & \cos(\alpha)(\cos(\alpha)-\theta) \\ 
    \cos(\alpha)(\cos(\alpha)-\overline\theta) &  \sin(\alpha)(2\,\mathrm{Re}(\theta)-\cos(\alpha))
  \end{array}
\right).
\end{align*}
satisfy $\kappa_1+\kappa_2=2\,\mathrm{Re}(\theta)\sin(\alpha)$ and $\kappa_1\kappa_2=2\cos(\alpha)(\mathrm{Re}(\theta)-\cos(\alpha))$. In particular
\begin{align*}
\ind(g_{e^{i\alpha+}})&=\coind(g_{e^{i\alpha+}})=1,\\
\ind(g_{e^{i\alpha}})&=\ind(g_{e^{i\alpha-}})=0,\\
\coind(g_{e^{i\alpha}})&=1,\\
\coind(g_{e^{i\alpha-}})&=2.
\end{align*}
The splitting and cosplitting numbers of $\Upsilon(1)$ at the eigenvalue $e^{i\alpha}$ are given by
\begin{align*}
\SSS_{\Upsilon(1)}^+(e^{i\alpha})=\coSSS_{\Upsilon(1)}^-(e^{i\alpha})=1,\\
\SSS_{\Upsilon(1)}^-(e^{i\alpha})=\coSSS_{\Upsilon(1)}^+(e^{i\alpha})=0.
\end{align*}
Recall that 
\begin{align*}
\SSS_{\Upsilon(1)}^+(\theta) & =\mas_{\theta^+}(\Upsilon)-\mas_{\theta}(\Upsilon),
\\
\coSSS_{\Upsilon(1)}^+(\theta) & =\comas_{\theta^+}(\Upsilon)-\comas_{\theta}(\Upsilon).
\end{align*}
Therefore
\begin{align*}
\mas_\theta(\Upsilon)
&=
\left\{
  \begin{array}{p{15pt}l}
    $-1$   & \mbox{if }\arg(\theta)\in[-\alpha,\alpha], \\ 
    0   & \mbox{otherwise,} \\ 
  \end{array}
\right.\\
\comas_\theta(\Upsilon)
&=
\left\{
  \begin{array}{p{15pt}l}
    $1$   & \mbox{if }\arg(\theta)\in(-\alpha,\alpha), \\ 
    0   & \mbox{otherwise.} \\ 
  \end{array}
\right.
\end{align*}
This computation, together with Bott's formulae~\eqref{e:Bott_formulae_Maslov}, allows us to compute the Maslov index and coindex of the original path $\Gamma=\Upsilon_p$. Indeed, consider the subsets of complex $p$-th roots of unity
\begin{align*}
\II:=\{ \theta\in\sqrt[p]1 \ |\ \arg(\theta)\in[-\alpha,\alpha] \}.
\end{align*}
Its cardinality is given by
\begin{align*}
|\II|
=
2 \left\lfloor \frac{\beta}{2\pi} \right\rfloor +1.
\end{align*}
Bott's formulae for the Maslov index give
\begin{align*}
\mas(\Gamma)=\sum_{\theta\in\sqrt[p]{1}} \mas_\theta(\Upsilon) = |\II|\,\mas(\Upsilon)
=
-2 \left\lfloor \frac{\beta}{2\pi} \right\rfloor -1.
\end{align*}
If $\beta$ is not a multiple of $2\pi$, Bott's formulae for the Maslov index give
\begin{align*}
\comas(\Gamma)=\sum_{\theta\in\sqrt[p]{1}} \comas_\theta(\Upsilon) = |\II|\,\comas(\Upsilon)
=
2 \left\lfloor \frac{\beta}{2\pi} \right\rfloor + 1,
\end{align*}
whereas if $\beta$ is a multiple of $2\pi$, i.e.\ $\Gamma(1)=I$, they give
\begin{align*}
\comas(\Gamma)&=\sum_{\theta\in\sqrt[p]{1}} \comas_\theta(\Upsilon)\\
& = (|\II|-2)\,\comas(\Upsilon)+2\,\comas_{e^{i\alpha}}(\Upsilon)\\
& = 2 \left\lfloor \frac{\beta}{2\pi} \right\rfloor - 1\\
\tag*{\qed} & = \frac{\beta}{\pi}-1.
\end{align*}
\end{exm}

We conclude this section by rephrasing the iteration inequalities of Theorem~\ref{t:iteration_inequalities} in the language of the Maslov index and coindex. In this form, the theorem is due to Liu and Long \cite{Liu_Long:An_optimal_increasing_estimate_of_the_iterated_Maslov-type_indices, Liu_Long:Iteration_inequalities_of_the_Maslov-type_index_theory_with_applications}.

\begin{thm}[Iteration inequalities for the Maslov indices]
\label{t:iteration_inequalities_Maslov}
Let $\Gamma:[0,1]\to\Sp(2d)$ be a continuous path such that $\Gamma(0)=I$, and let $p\in\N$. Then
\begin{equation}\label{e:iteration_inequality_Maslov}
\begin{split}
p\,\avmas(\zz) - d & \leq \mas(\Gamma_p),\\ 
\mas(\Gamma_p)+\dim\ker(\Gamma(1)^p-I) & \leq p\,\avmas(\Gamma) + d,
\end{split}
\end{equation}
where $\avmas(\Gamma)$ denotes the \textnormal{\textbf{average Maslov index}}, given by
\begin{align*}
\avmas(\Gamma):=\frac1{2\pi}\int_{0}^{2\pi} \mas_{e^{it}}(\Gamma)\,\diff t = \lim_{p\to\infty} \frac{\mas(\Gamma^p)}{p} \in\R.
\end{align*}
If at least one of the inequalities~\eqref{e:iteration_inequality_Maslov} is an equality, then $\sigma(\Gamma(1)^p)=\{1\}$ and $\dim\ker(\Gamma(1)^p-I)\geq d$. Both inequalities are equalities if and only if $\Gamma(1)^p=I$. \hfill\qed
\end{thm}

\begin{rem}
By~\eqref{e:relation_Maslov_coMaslov}, the inequalities~\eqref{e:iteration_inequality_Maslov} are equivalent to
\begin{align*}
p\,\avcomas(\zz) - d & \leq \comas(\Gamma_p),\\ 
\comas(\Gamma_p)+\dim\ker(\Gamma(1)^p-I) & \leq p\,\avcomas(\Gamma) + d,
\end{align*}
where $\avcomas(\Gamma)$ denotes the \textnormal{\textbf{average Maslov coindex}}, given by
\begin{align*}
\tag*{\qed}
\avcomas(\Gamma):=\frac1{2\pi}\int_{0}^{2\pi} \comas_{e^{it}}(\Gamma)\,\diff t = \lim_{p\to\infty} \frac{\comas(\Gamma^p)}{p} \in\R.
\end{align*}
\end{rem}

\subsection{Bibliographical remarks}

The Maslov index has quite a long history. It was first introduced by Gel'fand and Lidski\v{i} \cite{GelfandLidskii:On_the_structure_of_the_regions_of_stability_of_linear_canonical_systems_of_differential_equations_with_periodic_coefficients} as an index for the connected components of the space of strongly stable linear periodic Hamiltonian systems. It was later rediscovered by Maslov \cite{Maslov:Theorie_des_Perturbations_et_Methodes_Asymptotiques} as an intersection number of a loop of Lagrangian subspaces with the so called Maslov cycle, a singular hypersurface in the Lagrangian Grassmannian. Conley and Zehnder reinterpreted the Maslov index as a relative Morse index in \cite{Conley_Zehnder:Morse_type_index_theory_for_flows_and_periodic_solutions_for_Hamiltonian_equations}, and for this reason many authors in symplectic topology prefer the terminology \textbf{Conley-Zehnder index}. Our presentation of the Maslov index as a renormalized Morse index of Chaperon's generating families is analogous to Conley and Zehnder's one. This approach was already followed by Th\'eret~\cite[Chapter~IV]{Theret:Utilisation_des_fonctions_generatrices_en_geometrie_symplectique_globale} for more general generating families of Lagrangian submanifolds of cotangent bundles. Th\'eret inferred that the Maslov index is well defined (which is part of Theorem~\ref{t:Maslov_index_homotopy_invariance} above) as a consequence of  Viterbo's uniqueness Theorem for generating families \cite{Viterbo:Symplectic_topology_as_the_geometry_of_generating_functions, Theret:A_complete_proof_of_Viterbo_s_uniqueness_theorem_on_generating_functions}. An alternative proof of the relation between Maslov and Morse indices was provided by Robbin and Salamon \cite{Robbin_Salamon:Phase_functions_and_path_integrals}. In the references given so far, the Maslov index was considered only for ``non-degenerate'' paths, that is, for paths in $\PP_0(2d)$. The first author who defined the Maslov index on the whole space of symplectic paths $\PP(2d)$ was Long \cite{Long:Maslov_type_index_degenerate_critical_points_and_asymptotically_linear_Hamiltonian_systems}, who later on also defined the $\theta$-Maslov index and established its iteration theory \`a la Bott \cite{Long:Bott_formula_of_the_Maslov_type_index_theory}. Building on previous work of Conley and Zehnder, Long proved that the $\theta$-Maslov index classifies the path-connected components of the space $\PP_{\theta,0}(2d)$: two paths $\Gamma_1$ and $\Gamma_2$ belong to the same path-connected component of $\PP_{\theta,0}(2d)$ if and only if $\mas_\theta(\Gamma_1)=\mas_\theta(\Gamma_2)$ or, equivalently, $\comas_\theta(\Gamma_1)=\comas_\theta(\Gamma_2)$. We refer the reader to the monograph~\cite{Long:Index_theory_for_symplectic_paths_with_applications} for a comprehensive account of the $\theta$-Maslov index and for the many applications. A different extension of the Maslov index to degenerate paths, which is widely employed in symplectic topology, was given by Robbin and Salamon in \cite{Robbin_Salamon:The_Maslov_index_for_paths}.

\section{The Lagrangian Morse index}\label{s:Lagrangian}

\subsection{Tonelli Lagrangian and Hamiltonian systems}

In this section, we focus on a special class of Hamiltonian systems, for which the Maslov index can be described as a traditional Morse index of an action (without need of renormalization by a constant). This class can be described in the Lagrangian formulation as follows (we refer the reader to, e.g., \cite{Abraham_Marsden:Foundations_of_mechanics, Arnold:Mathematical_methods_of_classical_mechanics, Mazzucchelli:Critical_point_theory_for_Lagrangian_systems} for a comprehensive treatment of Lagrangian dynamics). Let $M$ be a manifold equipped with an auxiliary Riemannian metric. A \textbf{Tonelli Lagrangian} is a smooth time-dependent function $L_t:\Tan M\to\R$ such that $L_t=L_{t+1}$ and each function $v\mapsto L_t(q,v)$ has everywhere positive-definite Hessian and superlinear growth, i.e.
\begin{align*}
&\partial_{vv}^2 L_t(q,v)[w,w]>0,&& \forall t\in\R,\ (q,v)\in\Tan M,\ w\in\Tan_q M\setminus\{0\},\\
&\lim_{|v|_q\to\infty} L_t(q,v)/|v|_q = \infty,&& \forall t\in\R,\ q\in M.
\end{align*}
A Tonelli Lagrangian defines a second-order partial flow on $M$, that is, a flow on the tangent bundle $\Tan M$ whose integral lines are velocity vectors of curves on $M$. These curves $\gamma:(T_0,T_1)\to M$ are solution of the Euler-Lagrange equation
\begin{align*}
\tfrac{\diff}{\diff t} \partial_vL_t(\gamma(t),\dot\gamma(t))-\partial_qL_t(\gamma(t),\dot\gamma(t))=0.
\end{align*}
Let us assume for simplicity that the solutions of this equation are defined for all time. This is always true if $M$ is a closed manifold and the Lagrangian is autonomous, or more generally if its dependence on time is suitably controlled.

The fiberwise derivative $\partial_vL$ is a diffeomorphism of the tangent bundle $\Tan M$ onto the cotangent bundle $\Tan^*M$. The dual \textbf{Tonelli Hamiltonian} $H_t:\Tan^*M\to\R$ is defined by
\begin{align*}
H_t(q,p) = \max_{v\in \Tan_qM} \{pv-L_t(q,v)\}.
\end{align*}
This function still enjoys the Tonelli properties listed above: it is fiberwise convex and superlinear. Its fiberwise derivative $\partial_p H$ is the diffeomorphism inverse to $\partial_v L$, and we have
$L(q,v) + H(q,p) =pv$, where $p=\partial_vL(q,v)$ and $v=\partial_pH(q,p)$. The velocity curve $t\mapsto(\gamma(t),\dot\gamma(t))$ of a solution of the Euler-Lagrange equation is mapped by $\partial_vL$ to an integral curve $t\mapsto(\gamma(t),\partial_vL(\gamma(t),\dot\gamma(t)))$ of the Hamiltonian flow of $H$. We recall that the Hamiltonian flow $\phi_H^t$ is the integral of the non-autonomous Hamiltonian vector field $X_{H}$, which with our convention is defined by $\omega(X_{H_t},\cdot)=\diff H_t$, where $\omega=\diff q\wedge\diff p$ is the canonical symplectic form on the cotangent bundle $\Tan^*M$.

\subsection{The Lagrangian action functional}
A classical computation in calculus of variations shows that a smooth 1-periodic curve $\gamma:\R\to M$ is a solution of the Euler-Lagrange equation if and only if it is a critical point of the action functional $A:C^\infty(\R/\Z;M)\to\R$ given by
\begin{align*}
A(\gamma)=\int_0^1 L_t(\gamma(t),\dot\gamma(t))\,\diff t.
\end{align*}
We wish to investigate the properties of the Morse index $\ind(\gamma)$ of this functional at a critical point $\gamma$. For calligraphic convenience, let us assume that $M$ is the Euclidean space $\R^d$, so that the dual Hamiltonian $H$ defines a Hamiltonian flow on the standard symplectic $(\R^{2d},\omega)$. We associate to $\gamma$ the continuous path of symplectic matrices $\Gamma:[0,1]\to\Sp(2d)$ given by
\begin{align*}
\Gamma(t):=\diff\phi_H^t(\gamma(0),\partial_vL(\gamma(0),\dot\gamma(0))).
\end{align*}
Namely, $\Gamma$ is the path that begins at the identity matrix $\Gamma(0)=I$ and follows the linearized Hamiltonian flow at the starting point of the Hamiltonian periodic orbit corresponding to $\gamma$.

A priori, we do not know whether the Morse index $\ind(\gamma)$ is finite. This is a consequence of the following theorem, whose proof will be given at the end of Section~\ref{s:generating_family_Tonelli}, after several preliminaries.

\begin{thm}\label{t:Lagrangian_Morse}
The Morse index of $A$ at a critical point $\gamma$ coincides with the Maslov index of the associated symplectic path $\Gamma$, i.e.\
$\ind(\gamma)=\mas(\Gamma)$. 
\end{thm}

The Hessian of $A$ at $\gamma$ is the bilinear form on the infinite dimensional  Fr\'echet space $C^\infty(\R/\Z;\R^d)$ given by
\begin{align*}
\mathrm{Hess} A(\gamma)[\xi,\eta]
&=
\int_0^1
\Big( 
\langle \alpha\,\dot\xi,\dot\eta\rangle
+
\langle \beta\,\xi,\dot\eta\rangle
+
\langle \dot\xi,\beta\,\eta\rangle
+
\langle \delta\,\xi,\eta\rangle
\Big)\diff t,
\end{align*}
where
\begin{align*}
\alpha_t:=\partial_{vv}L_t(\gamma(t),\dot\gamma(t)),\quad
\beta_t:=\partial_{qv}L_t(\gamma(t),\dot\gamma(t)),\quad
\delta_t:=\partial_{qq}L_t(\gamma(t),\dot\gamma(t)).
\end{align*}
Since we are only interested in this Hessian form and not in the action functional $A$ itself, we can assume without loss of generality that the Lagrangian $L$ has the form
\begin{align}\label{e:quadratic_Lagrangian}
L_t(q,v)
=
\tfrac12
\langle \alpha_t v,v\rangle
+
\langle \beta_t q, v\rangle
+
\tfrac12
\langle \delta_t q,q\rangle,
\end{align}
and that $\gamma$ is the constant curve at origin. In this way, the Euler-Lagrange equation becomes linear of the form
\begin{align}\label{e:linear_Euler_Lagrange}
\alpha\, \ddot\xi + (\dot\alpha + \beta - \beta^T) \dot\xi + (\dot\beta-\delta)\xi =0,
\end{align}
and the action $A$ becomes a quadratic function, i.e.\ $A(\xi)=\tfrac12 \mathrm{Hess} A(\gamma)[\xi,\xi]$. From now on, we will simply write $\mathrm{Hess} A$ for $\mathrm{Hess} A(\gamma)$.

Let us extend $\mathrm{Hess} A$ as a bilinear form on the Sobolev space $W^{1,2}(\R/\Z;\R^d)$ of absolutely continuous curves with squared-integrable first derivative. One can show that the self-adjoint operator associated to this extension is Fredholm, and that the inertia index of the bilinear form is finite. Indeed, if the matrices $\beta_t$ and $\delta_t$ were identically zero, the Hessian form would clearly be semi-positive definite, since it would reduce to the integral 
\begin{align*}
\int_0^1
\langle \alpha\,\dot\xi,\dot\eta\rangle\,\diff t, 
\end{align*}
and $\alpha(t)$ is a positive definite matrix; the kernel of this bilinear form is given by the constant curves $\xi\equiv\xi(0)$, in particular it has finite dimension $d$. The general case, when $\beta_t$ and $\delta_t$ do not necessarily vanish identically, is a compact perturbation of this special one. When we add a compact perturbation to a semi-positive definite Fredholm bilinear form, the index of the resulting form is finite (see e.g.\ \cite[Lemma~2.1.2 and errata corrige]{Mazzucchelli:Critical_point_theory_for_Lagrangian_systems}). Therefore, $\mathrm{Hess} A$ has finite index. Since $C^\infty(\R/\Z;\R^d)$ is dense in $W^{1,2}(\R/\Z;\R^d)$, one can show that the Morse index of the Hessian is the same whether we consider it as a bilinear form on $C^\infty(\R/\Z;\R^d)$ or on $W^{1,2}(\R/\Z;\R^d)$.

Consider now, for each integer $k\geq2$, the vector space
\begin{align*}
\EE_k:= \big\{ \xi\in C^0(\R/\Z;\R^d)\ \big|\ \xi|_{[j/k,(j+1)/k]}\mbox{ is a solution of \eqref{e:linear_Euler_Lagrange} }\forall j\in\Z_{k} \big\}.
\end{align*}
Notice that $\EE_k$ has finite dimension $dk$, and the evaluation map 
\[\xi\mapsto(\xi(0),\xi(1/k),...,\xi((k-1)/k))\] 
is an isomorphism of $\EE_k$ onto $\R^{dk}$. Moreover $\EE_k\subset\EE_{2k}$. As $k$ increases, $\EE_k$ contains finer and finer approximations of any given smooth $1$-periodic curve.  Actually, one can show that the union of all the $\EE_k$'s is dense in the Sobolev space $W^{1,2}(\R/\Z;\R^d)$, and therefore that
\begin{align*}
\ind(\mathrm{Hess}A) = \ind(\mathrm{Hess}A|_{\EE_k\times\EE_k}),\qquad\forall k\geq2\mbox{ large enough}.
\end{align*}
We recall that the action $A$ is assumed to be a quadratic function. In particular, a curve $\xi$ is in the kernel of $\mathrm{Hess}A$ if and only if it is a critical point of $A$, that is, if and only if it is a solution of the Euler-Lagrange equation~\eqref{e:linear_Euler_Lagrange}. Therefore
\begin{align*}
 \ker(\mathrm{Hess}A) = \ker(\mathrm{Hess}A|_{\EE_k\times\EE_k}),\qquad\forall k\geq2.
\end{align*}
For more details on this, we refer the reader to~\cite[Section~4.4]{Mazzucchelli:Critical_point_theory_for_Lagrangian_systems}.

Let us have a look at the expression of the Hessian of $A$ on the space $\EE_k$. For all $\xi,\eta\in\EE_k$, we have
\begin{align*}
\mathrm{Hess} A[\xi,\eta] 
=\,& 
\int_0^1
\Big( 
\langle \alpha\, \dot\xi,\dot\eta\rangle
+
\langle \beta\, \xi,\dot\eta\rangle
+
\langle \dot\xi,\beta\, \eta\rangle
+
\langle \delta\, \xi,\eta\rangle
\Big)\,\diff t\\
=\,&
\sum_{j=0}^{k-1}
\int_{j/k}^{(j+1)/k}
\langle 
\underbrace{\big. -\alpha\, \ddot\xi - (\dot\alpha + \beta - \beta^T) \dot\xi - (\dot\beta-\delta)\xi}_{=0}
,\eta\rangle\,\diff t\\
& + \sum_{j=0}^{k-1} \langle \alpha\dot\xi + \beta\xi,\sigma\rangle \Big|_{j^+/k}^{(j+1)^-/k}\\
=\,&
\sum_{j=0}^{k-1}
\langle \alpha_{j/k}\big(\dot\xi(\tfrac{j}{k}^-)-\dot\xi(\tfrac{j}{k}^+)\big),\sigma(\tfrac{j}{k})\rangle
\end{align*}
It will be more convenient to write down this Hessian in a slightly different way as follows. We denote by $\phi_L^t$ the Euler-Lagrange flow on the tangent bundle $\Tan\R^d=\R^{2d}$, which is defined by $\phi_L^t(\xi(0),\dot\xi(0))=(\xi(t),\dot\xi(t))$ if $\xi:[0,t]\to\R^d$ is a solution of the Euler-Lagrange equation. We set
\begin{align*}
Q_j:= \phi_L^{(j+1)/k}\circ(\phi_L^{j/k})^{-1},\qquad\forall j=0,...,k-1,
\end{align*}
so that $\phi_L^{j/k}=Q_j\circ...\circ Q_0$, and we denote by $\pi_1:\R^{2d}\to\R^d$ the projection $\pi(X,V)=X$. We introduce the vector space
\begin{align*}
\VV:=
\left\{
(X_0,V_0,...,X_{k-1},V_{k-1})\in \R^{2dk}
\, \big|\ 
    \pi_1\circ Q_j (X_j,V_j) =X_{j+1}\quad\forall j\in\Z_k\right\}.
\end{align*}
Notice that there is an isomorphism $\Psi:\EE_k\to\VV$ given by 
\[\Psi(\xi)=\big(\xi(0),\dot\xi(0^+),\xi(\tfrac{1}{k}),\dot\xi(\tfrac{1}{k}^+),...,\xi(\tfrac{k-1}{k}),\dot\xi(\tfrac{k-1}{k}^+)\big).\]
If we pull-back the Hessian of the action $A$ by the isomorphism $\Psi^{-1}$, we obtain the simmetric bilinear form $h_L:\VV\times\VV\to\R$ that reads
\begin{align*}
 h_L(\ZZ,\ZZ')=\mathrm{Hess}A[\Psi^{-1}\ZZ,\Psi^{-1}\ZZ']
 =
\sum_{j\in\Z_k}
\langle \alpha_{j/k} (\tilde V_{j} - V_{j}),X_j'\rangle,
\end{align*}
where $\ZZ=(X_0,V_0,...,X_{k-1},V_{k-1})$, $\ZZ'=(X_0',V_0',...,X_{k-1}',V_{k-1}')$, and we have adopted the notation $(X_{j+1},\tilde V_{j+1})=Q_j(X_j,V_j)$. Summing up, in order to prove Theorem~\ref{t:Lagrangian_Morse}, we have to show that 
\begin{align}\label{e:Lagrangian_index_to_be_proved}
\ind(h_L)=\mas(\Gamma). 
\end{align}

\subsection{The generating family of a Tonelli Hamiltonian flow}\label{s:generating_family_Tonelli}

Let us now focus on the linear Hamiltonian flow $\phi_H^t$, which we discretize by setting
\begin{align*}
P_j:= \phi_H^{(j+1)/k}\circ(\phi_H^{j/k})^{-1},\qquad\forall j=0,...,k-1.
\end{align*}
Notice that the matrices $P_j$ are related to the matrices $Q_j$ of the previous subsection by 
\begin{align}\label{e:conjugacy_Lagrangian_Hamiltonian}
P_j\circ\partial_v L_{j/k}=\partial_v L_{(j+1)/k} \circ Q_j,
\end{align}
and $\partial_vL_t(x,v)=(x,\alpha_t v + \beta_t q)$. Since our parameter $k$ is assumed to be large enough, each symplectic matrix $P_j\in\Sp(2d)$ is close to the identity, and therefore admits a quadratic generating function
\begin{align*}
f_j(X_{j+1},Y_j)=
\tfrac12\langle A_j X_{j+1},X_{j+1}\rangle
+
\langle B_j X_{j+1},Y_j\rangle
+
\tfrac12\langle C_j Y_j,Y_j\rangle,
\end{align*}
where $A_j$, $B_j$, and $C_j$ are (small) $dk\times dk$ real matrices, $A_j$ and $C_j$ being symmetric. As we know, this means that
\begin{align*}
P_j Z_j=Z_{j+1}\qquad\mbox{if and only if}\qquad
\left\{
  \begin{array}{l}
    X_{j+1}-X_j=-B_j X_{j+1} - C_j Y_j, \\ 
    Y_{j+1}-Y_j=A_j X_{j+1}+B_j^T Y_j .
  \end{array}
\right.
\end{align*}
Let us show the precise relationship between the Hamiltonian $H$ and the generating functions $f_j$.

\begin{lem}\label{l:generating_function_and_hamiltonian}
If $(X(t),Y(t)):=\phi_H^t(X(0),Y(0))$ is an orbit of the Hamiltonian flow and we set $(X_j,Y_j):=(X(j/k),Y(j/k))$, we have
\begin{align*}
f_j(X_{j+1},Y_j)=
\langle Y_j,X_j-X_{j+1} \rangle + \int_{j/k}^{(j+1)/k} \Big( \langle Y(t),\dot X(t)\rangle -H_t(X(t),Y(t))\Big)\,\diff t.
\end{align*}
\end{lem}

\begin{proof}
For syntactic convenience, let us focus on the case $j=0$, the other cases being completely analogous. Consider the primitive $-y\,\diff x$ of the symplectic form $\omega=\diff x\wedge\diff y$. Since the Hamiltonian flow $\phi_H^t$ is symplectic,  $(\phi_H^t)^*y\,\diff x-y\,\diff x$ is a closed 1-form, hence exact by the Poincar\'e Lemma. Let $g_0:\R^{2d}\to\R$ be a function defined up to an additive constant by
\begin{align}\label{e:action_Hamiltonian}
\diff g_0=(\phi_H^{1/k})^*y\,\diff x-y\,\diff x. 
\end{align}
By applying the Fundamental Theorem of Calculus to the right-hand side of this equation, we obtain
\begin{align*}
\diff g_0& =(\phi_H^{1/k})^*y\,\diff x-y\,\diff x \\
& = \int_0^{1/k} (\phi_H^t)^* \mathcal{L}_{X_{H_t}} (y\,\diff x)\,\diff t\\
& = \int_0^{1/k} (\phi_H^t)^* (\diff (y\,\diff x(X_{H_t})  )-\omega(X_{H_t},\cdot))\,\diff t\\
& = \diff \left( \int_0^{1/k} (\phi_H^t)^* ( y\,\diff x(X_{H_t})  - H_t)\,\diff t \right) \\
\end{align*}
If we normalize $g_0$ by setting $g_0(0)=0$, we have 
\begin{align*}
g_0=\int_0^{1/k} (\phi_H^t)^* ( y\,\diff x(X_{H_t})  - H_t)\,\diff t.
\end{align*}
By evaluating this expression at the starting point $(X_0,Y_0)$ of our orbit, we obtain
\begin{align*}
g_0(X_0,Y_0) =  \int_0^{1/k} \Big(\langle Y(t),\dot X(t)\rangle - H_t(X(t),Y(t))\Big)\,\diff t.
\end{align*}
Notice that $g_0$ is a quadratic function.

Now, let us consider $X_1$ and $Y_0$ as independent variables, while $X_0=X_0(X_1,Y_0)$ and $Y_1=Y_1(X_1,Y_0)$. More precisely, we denote by $R_0:\R^2\to\R^2$ the linear isomorphism such that $P_0(X_0,Y_0)=(X_1,Y_1)$ if and only if $R_0(X_1,Y_0)=(X_0,Y_0)$. Equation~\eqref{e:action_Hamiltonian} becomes
\begin{align*}
\diff(g_0\circ R_0)&=R_0^*(\diff g_0)\\
&=Y_1\,\diff X_1 - Y_0\,\diff X_0\\
&=(Y_1-Y_0)\,\diff X_1 - Y_0\,(\diff X_0 -\diff X_1)\\
&=\underbrace{(Y_1-Y_0)}_{\partial_{X_1}f_0}\,\diff X_1 + \underbrace{(X_0-X_1)}_{\partial_{Y_0}f_0}\,\diff Y_0 - \diff( \langle Y_0, X_0 -X_1\rangle ) \\
&=\diff f_0 - \diff( \langle Y_0, X_0 -X_1\rangle ).
\end{align*}
This defines the generating function $f_0$ up to a constant. Since $f_0$ is a quadratic function, it vanishes at the origin, and therefore we conclude 
\[
f_0= \langle Y_0, X_0 -X_1\rangle + g_0\circ Q_0.
\qedhere
\]
\end{proof}

\begin{rem}
The above proof works with any (not necessarily linear) Hamiltonian flow, except that the functions $f_0$ and $g_0$ are not quadratic anymore and therefore can only  be defined up to an additive constant.
\end{rem}

Lemma~\ref{l:generating_function_and_hamiltonian} allows us to translate the Tonelli fiberwise convexity property of the Hamiltonian $H$ to a concavity property for the generating functions $f_j$.

\begin{lem}\label{l:concavity_generating_function}
If the parameter $k$ is large enough, each matrix $C_j$ is negative definite.
\end{lem}

\begin{proof}
Let us compute the explicit expression of our Hamiltonian $H$ dual to the quadratic Lagrangian~\eqref{e:quadratic_Lagrangian}. Given $(q,v)\in\R^{2d}$, the dual moment variable $p$ is given by
\begin{align*}
(q,p) = \partial_v L_t(q,v) = (q,\alpha_t v + \beta_t q).
\end{align*}
Therefore
\begin{align*}
H_t(q,p)&=pv-L(q,v)\\
&=\langle p,\alpha^{-1}(p-\beta_t q)\rangle - L(q,\alpha^{-1}(p-\beta_t q))\\
&=\tfrac12 \langle \alpha^{-1}p,p\rangle - \langle \alpha_t^{-1} \beta_t q,p \rangle +\tfrac12 \langle (\beta_t^T\alpha_t^{-1}\beta_t - \delta_t) q,q \rangle.
\end{align*}
Let $Y_j\in\R^{d}$ and $X_j:=C_jY_j$, so that $P_j(X_j,Y_j)=(0,Y_{j+1})$. By Lemma~\ref{l:generating_function_and_hamiltonian}, we have
\begin{align*}
\langle C_j Y_j,Y_j\rangle 
& = 2f_j(0,Y_j)\\
& =
2 \langle Y_j,X_j \rangle + 2\int_{j/k}^{(j+1)/k} \Big( \langle Y(t),\dot X(t)\rangle -H_t(X(t),Y(t))\Big)\,\diff t\\
& =
 2\int_{j/k}^{(j+1)/k} \Big( -\langle \dot Y(t),X(t)\rangle -H_t(X(t),Y(t))\Big)\,\diff t\\
& =
 2\int_{j/k}^{(j+1)/k} \Big(  \partial_q H(X(t),Y(t))\,X(t)  -H_t(X(t),Y(t))\Big)\,\diff t\\
& =
 \int_{j/k}^{(j+1)/k} \Big( -\langle\alpha_t^{-1}Y(t),Y(t)\rangle +\langle(\beta_t^T\alpha_t^{-1}\beta_t - \delta_t)X(t),X(t)\rangle \Big)\,\diff t\\
& \leq
- a \int_{j/k}^{(j+1)/k} |Y(t)|^2\,\diff t + b \int_{j/k}^{(j+1)/k} |X(t)|^2\,\diff t,
\end{align*}
where 
\begin{align*}
 a:=\min_{t\in\R/\Z} \big|\alpha_t^{-1}\big|>0,\qquad
 b:=\max_{t\in\R/\Z} \big|\beta_t^T\alpha_t^{-1}\beta_t - \delta_t\big|.
\end{align*}
We recall that the Hamiltonian flow $\phi_H^t$ is linear. For all $\epsilon>0$ there exists $k\in\N$ large enough such that, for all $t_1,t_2\in[0,1]$ with $|t_1-t_2|\leq1/k$, we have
\begin{align*}
|\phi_H^{t_1}\circ(\phi_H^{t_2})^{-1}-I|<\epsilon.
\end{align*}
In other words, if $t\mapsto Z(t)=(X(t),Y(t))$ is a non-zero integral curve of the  Hamiltonian flow $\phi_H^t$, we have
\begin{align*}
\frac{|Z(t_1)-Z(t_2)|}{|Z(t_2)|} < \epsilon,\qquad\forall t_1,t_2\in[0,1] \mbox{ with } |t_1-t_2|\leq 1/k,
\end{align*}
and if we further assume that $X(t_2)=0$, we infer
\begin{align*}
|X(t_1)|<\epsilon |Y(t_2)|,\qquad
|Y(t_1)|>(1-\epsilon) |Y(t_2)|.
\end{align*}
By plugging these inequalities into the estimate for $\langle C_jY_j,Y_j\rangle$ above, we obtain
\begin{align*}
\langle C_jY_j,Y_j\rangle \leq -a\frac{(1-\epsilon)^2 }{k} |Y_{j+1}|^2
+ b \frac{\epsilon^2}{k}|Y_{j+1}|^2
=
\underbrace{( - a(1-\epsilon)^2+b\epsilon^2)}_{(*)}
\frac{|Y_{j+1}|^2}{k},
\end{align*}
and the term $(*)$ is negative provided $\epsilon$ is small enough.
\end{proof}

\begin{proof}[Proof of Theorem~\ref{t:Lagrangian_Morse}]
Consider the quadratic generating family $F:\R^{2dk}\to\R$ associated to the factorization $\phi_H^1=P_{k-1}\circ...\circ P_0$. We recall that the Hessian bilinear form $h:\R^{2dk}\times\R^{2dk}\to\R$ of $F$ is given by
\begin{align*}
h(\ZZ,\ZZ')
= & \,
 \sum_{j\in\Z_k} \langle Y_{j-1} - Y_{j} + A_{j-1} X_j + B_{j-1}^T Y_{j-1} , X_j'\rangle \\
 & + \sum_{j\in\Z_k} \langle X_{j+1} - X_{j} + B_{j} X_{j+1} + C_{j} Y_{j},Y_j'\rangle.
\end{align*}
We wish to take advantage of the fact that the matrices $C_j$ are negative definite (Lemma~\ref{l:concavity_generating_function}) in order to compute the Morse index of $h$ and, a fortiori, the Maslov index of $\Gamma$. We introduce the vector subspace
\begin{align*}
\WW:=
\big\{
(X_0,Y_0,...,X_{k-1},Y_{k-1})\in\R^{2dk}\ \big|\ X_j=0\quad\forall j=0,...,k-1
\big\},
\end{align*}
and its $h$-orthogonal
\begin{align*}
\WW^h=
\left\{
(X_0,Y_0,...,X_{k-1},Y_{k-1})\in\R^{2dk}\ \left|\ 
\begin{array}{r}
X_{j+1} - X_{j} + B_{j} X_{j+1} + C_{j} Y_{j}=0 \vspace{5pt}\\ 
\forall j=0,...,k-1
\end{array}
\right.\right\},
\end{align*}
For all $\ZZ,\ZZ'\in\WW$, we have
\begin{align*}
h(\ZZ,\ZZ')=\sum_{j\in\Z_k} \langle C_j Y_j,Y_j'\rangle.
\end{align*}
Therefore, $h|_{\WW\times\WW}$ is a negative definite bilinear form, and in particular
\begin{align*}
\ind(h|_{\WW\times\WW})=\dim\WW=dk.
\end{align*}
The intersection $\WW\cap\WW^h$ is given by those vectors $\ZZ\in\R^{2dk}$ such that  $X_j=C_jY_j=0$ for all $j=0,...,k-1$. Since the matrices $C_j$ are negative definite, they are invertible, and therefore the intersection $\WW\cap\WW^h$ is trivial. Since $\WW\cap\ker(h)$ is contained in $\WW\cap\WW^h$, it is trivial as well. By applying Proposition~\ref{p:index_restricted_form}, we obtain
\begin{align*}
\ind(h)
=\ind(h|_{\WW\times\WW}) + \ind(h|_{\WW^h\times\WW^h}) 
= dk + \ind(h|_{\WW^h\times\WW^h}).
\end{align*}
By rephrasing in terms of the Maslov index of the path $\Gamma$, we have
\begin{align*}
\mas(\Gamma)=\ind(h)-dk =\ind(h|_{\WW^h\times\WW^h}).
\end{align*}
Let us now focus on the form $h|_{\WW^h\times\WW^h}$. We denote by $\pi_1:\R^{2d}\to\R^d$ the projection $\pi_1(X,Y)=X$. Notice that the vector space $\WW^h$ can be characterized as
\begin{align*}
\WW^h:=
\left\{
(X_0,Y_0,...,X_{k-1},Y_{k-1})\in \R^{2dk}
\, \big|\ 
    \pi_1\circ P_j (X_j,Y_j) =X_{j+1}\quad\forall j\in\Z_k\right\}.
\end{align*}
In particular, $\WW^h$ is isomorphic to the vector space $\VV$ of the previous subsection via the isomorphism $\Omega:\VV\to\WW^h$ given by
\begin{align*}
 \Omega(X_0,V_0,...,X_{k-1},V_{k-1})=(X_0,Y_0,...,X_{k-1},Y_{k-1}),
\end{align*}
where 
\[(X_j,Y_j)=\partial_vL_{j/k}(X_j,V_j)=(X_j,\alpha_{j/k}V_j + \beta_{j/k}X_j).\] 
We also set
\begin{align*}
 \tilde Y_j := Y_{j-1} + A_{j-1} X_j + B_{j-1}^T Y_{j-1},\qquad\forall j\in\Z_k,
\end{align*}
so that $P_j(X_j,Y_j)=(X_{j+1},\tilde Y_{j+1})$. We recall the notation of the previous subsection: we write $Q_j(X_j,V_j)=(X_{j+1},\tilde V_{j+1})$. Equation~\eqref{e:conjugacy_Lagrangian_Hamiltonian} implies that the vectors $\tilde V_j$ and $\tilde Y_j$ are related by the usual duality 
\[(X_j,\tilde Y_j)=\partial_vL_{j/k}(X_j,\tilde V_j)=(X_j,\alpha_{j/k}\tilde V_j + \beta_{j/k}X_j).\]
For all $\ZZ,\ZZ'\in\WW^h$, we have
\begin{align*}
h(\ZZ,\ZZ')
& = 
 \sum_{j\in\Z_k} \langle Y_{j-1} - Y_{j} + A_{j-1} X_j + B_{j-1}^T Y_{j-1} , X_j'\rangle \\
& = 
 \sum_{j\in\Z_k} \langle \tilde Y_{j} - Y_{j}, X_j'\rangle.
\end{align*}
If we pull-back $h|_{\WW^h\times\WW^h}$ by the isomorphism $\Omega$, we obtain
\begin{align*}
h(\Omega \ZZ,\Omega \ZZ')
& = \sum_{j\in\Z_k} \langle \alpha_{j/k}\tilde V_j + \beta_{j/k}X_j - \alpha_{j/k} V_j - \beta_{j/k}X_j, X_j'\rangle\\
& = \sum_{j\in\Z_k} \langle \alpha_{j/k}( \tilde V_j - V_j) , X_j'\rangle\\
& = h_L(\ZZ,\ZZ').
\end{align*}
In particular $\ind(h|_{\WW^h\times\WW^h})=\ind(h_L)$. This completes the proof of~\eqref{e:Lagrangian_index_to_be_proved}, and thus of Theorem~\ref{t:Lagrangian_Morse}. 
\end{proof}

\subsection{Bibliographical remarks}

Historically, the first statement of the kind of Theorem~\ref{t:Lagrangian_Morse} above is the Index Theorem from Riemannian geometry \cite[Section~15]{Milnor:Morse_theory}, asserting that the Morse index of a geodesic with prescribed endpoints is given by its number of conjugate points counted with multiplicity. Indeed, this count corresponds to the Maslov index of an associated path of Lagrangian subspaces. The periodic orbit case for Tonelli Lagrangian systems was first established by Duistermaat \cite{Duistermaat:On_the_Morse_index_in_variational_calculus}. The proof that we provided in this section is conceptually similar to the one given by Abbondandolo in~\cite{Abbondandolo:On_the_Morse_index_of_Lagrangian_systems}. See also~\cite{Viterbo:Intersection_de_sous_varietes_lagrangiennes_fonctionnelles_d_action_et_indice_des_systemes_hamiltoniens, An_Long:On_the_index_theories_for_second_order_Hamiltonian_systems, Abbondandolo:Morse_theory_for_Hamiltonian_systems, Long:Index_theory_for_symplectic_paths_with_applications} for other proofs and related results.

\appendix
\section{Some linear algebra}

\subsection{Eigenspaces of power matrices}

A non-diagonalizable squared complex matrix $M$ must have an eigenvalue $\lambda$ whose algebraic multiplicity is strictly larger than its geometric one. If $\lambda=0$ with algebraic multiplicity $n$, then the algebraic multiplicity of the eigenvalue $\lambda$ becomes equal to its geometric one for the power matrix $M^n$. The next proposition shows that this never occurs for non-zero eigenvalues.

\begin{prop}\label{p:power_matrix}
For every squared complex matrix $M$ we have
\begin{align*}
 \dim_\C\ker(M^n-\theta I) = \sum_{\mu\in\sqrt[n]\theta} \dim_\C\ker(M-\mu I),\qquad\forall n\in\N,\ \theta\in\C\setminus\{0\}.
\end{align*}
\end{prop}

\begin{proof}
Assume without loss of generality that $M$ is in Jordan normal form, with Jordan blocks $M_1,...,M_r$. Hence, its $n$-th power $M^n$ is a block-diagonal matrix with blocks $M_1^n,...,M_r^n$, and since
\[
\dim_\C\ker(M^n-\theta I)=\sum_{j=1}^r \dim_\C\ker(M_j^n-\theta I),
\]
it suffices to prove the proposition for the case in which $M=M_1$ is a single  Jordan block with eigenvalue $\mu\neq0$, i.e.
\[
M
=
\left(
\begin{matrix}
\mu &1 \\
& \mu & 1\\
&& \ddots & \ddots\\
&&&\mu & 1\\
&&&&\mu
\end{matrix}
\right).
\]
In this case, the claim of the proposition reduces to
\begin{align*}
\dim_\C\ker(M^n-\mu^n I) = \dim_\C\ker(M-\mu I)=1.
\end{align*}
By a straightforward computation, we can verify that the power matrix $M^n$ is still upper-triangular, where the entries in the diagonal are all equal to $\mu^n$, while the entries in the super-diagonal are all equal to $n\,\mu^{n-1}$. The matrix $M^n-\mu^nI$ is upper-triangular, with entries in the diagonal all equal to zero, and entries in the super-diagonal all equal to $n\,\mu^{n-1}$. In particular, the first column of $M^n-\mu^nI$ is the zero one, while the other columns are linearly independent. This proves that the kernel of $M^n-\mu^nI$ is one-dimensional.
\end{proof}

\subsection{Inertia of restricted Hermitian forms}\label{a:restriction_quadratic_forms}

Let $H$ be a Hermitian $d\times d$ matrix, and $h:\C^d\times\C^d\to\C$ the associated Hermitian form $h(v,w)=\langle Hv,w\rangle$. We recall the definition of the \textbf{inertia} triple of $h$: the \textbf{index} $\ind(h)$ equal to the maximal dimension of a vector subspace over which $h$ is negative definite, the \textbf{coindex} $\coind(h)=\ind(-h)$ equal to the maximal dimension of a vector subspace over which $h$ is positive definite, and the \textbf{nullity} $\nul(h)$ equal to the dimension of the kernel of $h$, that is, the kernel of the matrix $H$. Of course, in a Hermitian setting, dimension will always stand for complex dimension. If the matrix $H$ is real, the exact same results of this section hold for the real simmetric bilinear form $h|_{\R^d\times\R^d}$ by replacing complex dimension with real dimension in all the formulae (as well as in the definition of index, coindex, and nullity of $h|_{\R^d\times\R^d}$).

Given a complex vector subspace $\VV\subseteq\C^d$, its $h$-orthogonal is the complex vector subspace defined by
\begin{align*}
\VV^h=\big\{ w\in\C^d\ \big|\ h(w,v)=0\quad \forall v\in \VV \big\}.
\end{align*}
It readily follows from its definition that 
\begin{align*}
\VV^h=(H\VV)^\bot=H^{-1}(\VV^\bot), 
\end{align*}
where $\bot$ denotes the orthogonal with respect to the Hermitian inner product $\langle\cdot,\cdot\rangle$. Moreover,
\begin{align*}
(\VV^h)^h= H^{-1}H\VV = \VV + \ker(H).
\end{align*}

The inertia of $h$ is related to the one of the restricted forms $h|_{\VV\times\VV}$ and $h|_{\VV^h\times\VV^h}$ according to the following statements.

\begin{prop}\label{p:nul_restricted_form}
$\nul(h)=\nul(h|_{\VV^h\times\VV^h}) - \dim_\C(\VV\cap\VV^h) + \dim_\C(\VV\cap\ker(H))$.
\end{prop}

\begin{proof}
The kernel of the Hermitian matrix associated to the restricted form $h|_{\VV^h\times\VV^h}$ is given by
\begin{align*}
\ker(h|_{\VV^h\times\VV^h})&=\big\{v\in \VV^h\ \big|\   Hv\in ((H\VV)^\bot)^\bot  \big\}\\
&=
\big\{v\in \VV^h\ \big|\   Hv\in H\VV  \big\}\\
&=
\big\{v\in \VV^h\ \big|\   v\in H^{-1}H\VV  \big\}\\
&= (\VV+\ker(H))\cap \VV^h.
\end{align*}
Notice that $\ker(H)\subset\VV^h$. Moreover
\begin{align*}
\dim_\C( (\VV+\ker(H))\cap \VV^h) &=\dim_\C(\VV+\ker(H)) + \dim_\C(\VV^h)\\
&\quad -\dim_\C(\VV + \ker(H) + \VV^h)\\
&=\dim_\C(\VV)+\dim_\C\ker(H)-\dim_\C(\VV\cap\ker(H))\\
&\quad+ \dim_\C(\VV^h) -\dim_\C(\VV + \VV^h)\\
&= \dim_\C(\VV\cap\VV^h)+\nul(h)-\dim_\C(\VV\cap\ker(H)).
\end{align*}
These two equations prove the proposition.
\end{proof}

\begin{prop}\label{p:index_restricted_form}
\begin{align*}
\ind(h) & = \ind(h|_{\VV\times\VV}) + \ind(h|_{\VV^h\times\VV^h}) + \dim_\C(\VV\cap\VV^h) - \dim_\C(\VV\cap\ker(H)),\\
\coind(h) & = \coind(h|_{\VV\times\VV}) + \coind(h|_{\VV^h\times\VV^h}) + \dim_\C(\VV\cap\VV^h) - \dim_\C(\VV\cap\ker(H)).
\end{align*}
\end{prop}

\begin{proof}
Since the coindex of a quadratic form is equal to the index of minus the same quadratic form, it is enough to prove the equality for the index. We first give a proof in case $h$ is a non-degenerate bilinear form, that is, in case the associated Hermitian matrix $H$ is invertible. Under this assumption, the last summand on the right-hand side of the equality that we want to prove is zero. Moreover
\begin{align*}
\VV\cap\VV^h=\ker(h|_{\VV\times\VV})=\ker(h|_{\VV^h\times\VV^h}).
\end{align*}
The restricted Hermitian form $h|_{\VV\times\VV}$ can be written as
\begin{align*}
 h|_{\VV\times\VV}(v,w)=\langle P_\VV\circ H|_{\VV} v,w\rangle,
\end{align*}
where $P_\VV:\C^d\to\VV$ is the orthogonal projector onto $\VV$, which is an Hermitian linear map. Notice that $P_\VV\circ H|_{\VV}$ is Hermitian. In particular it is diagonalizable and has only real eigenvalues. Therefore, the vector subspace $\VV$ splits as the direct sum \[\VV=\EE^-\oplus\EE^+\oplus (\VV\cap\VV^h),\] where $\EE^-$ is the direct sum of the eigenspaces of $P_\VV\circ H|_{\VV}$ corresponding to negative eigenvalues, while $\EE^+$ is the direct sum of the eigenspaces corresponding to positive eigenvalues. These three vector spaces in the direct-sum decomposition of $\VV$ are orthogonal with respect to both the Hermitian inner product $\langle\cdot,\cdot\rangle$ and the Hermitian form $h|_{\VV\times\VV}$. The inertia of this latter form is precisely
\begin{align*}
\ind(h|_{\VV\times\VV})&=\dim_\C\EE^-,\\
\coind(h|_{\VV\times\VV})&=\dim_\C\EE^+,\\
\nul(h|_{\VV\times\VV})&=\VV\cap\VV^h.
\end{align*}
Since $\EE^+$ and $\EE^-$ are invariant by the linear map $P_\VV\circ H$,  we have that $H(\EE^\pm)\subset \EE^\pm + \VV^\bot$. Let us introduce an analogous splitting \[\VV^h=\FF^+\oplus\FF^-\oplus(\VV\cap\VV^h),\] where $\FF^-$  and $\FF^+$ are the direct sum of the eigenspaces of $P_{\VV^h}\circ H|_{\VV^h}$ corresponding to the negative eigenvalues and to the positive eigenvalues respectively. As before, we have
\begin{align*}
\ind(h|_{\VV^h\times\VV^h})&=\dim_\C\FF^-,\\
\coind(h|_{\VV^h\times\VV^h})&=\dim_\C\FF^+,\\
\nul(h|_{\VV^h\times\VV^h})&=\VV\cap\VV^h.
\end{align*}
Notice that the vector subspaces $\EE^-$ and $\FF^-$ are $h$-orthogonal, and in particular the form $h$ is negative definite on the subspace $\EE^-\oplus\FF^-$. The analogous consideration holds for the vector subspaces $\EE^+$ and $\FF^+$. The matrix $H$ maps the intersection $\VV\cap\VV^h$ isomorphically onto $\VV^\bot\cap(\VV^h)^\bot=(\VV+\VV^h)^\bot$. We fix a real constant $\lambda\in(0,2/\|H^{-1}\|)$, and we introduce the vector subspaces
\begin{align*}
\GG^- & :=\{v-\lambda H^{-1}v \ |\ v\in \VV \cap \VV^h\},\\
\GG^+ & :=\{v+\lambda H^{-1}v \ |\ v\in \VV \cap \VV^h\}.
\end{align*}
The form $h$ is negative definite on $\GG^-$. Indeed, for all $v\in \VV \cap \VV^h$,
\begin{align*}
h(v-\lambda H^{-1}v,v-\lambda H^{-1}v) &= h(v,v) -2\lambda h(H^{-1}v,v) + \lambda^2 h(H^{-1}v,H^{-1}v)\\
&= -2\lambda \|v\|^2 +\lambda^2\langle v,H^{-1}v\rangle\\
&\leq \lambda \|v\|^2 \underbrace{\big( -2 + \lambda \|H^{-1}\| \big)}_{<0}.
\end{align*}
Analogously, $h$ is positive definite on $\GG^+$. The vector spaces $\GG^\pm$ are $h$-orthogonal to $\EE^\pm\oplus\FF^\pm$, since for all $v\in \VV \cap \VV^h$ and $w\oplus z\in\EE^\pm\oplus\FF^\pm$ we have
\begin{align*}
h(w+z,v\pm\lambda H^{-1}v) & = h(w,v) +h(z,v) \pm \lambda\, h(w,H^{-1}v) \pm \lambda\, h(z,H^{-1}v)\\
& = \pm \lambda \langle w,v\rangle \pm \lambda \langle z,v\rangle\\
& = 0.
\end{align*}
We conclude that $h$ is negative definite on $\EE^-\oplus\FF^-\oplus\GG^-$ and positive definite on $\EE^+\oplus\FF^+\oplus\GG^+$. Since the direct sum of these two vector subspaces is the whole $\C^d$, we have that
\begin{equation}\label{e:restricted_index_nondegenerate_case}
\begin{split}
\ind(h) & =\dim_\C(\EE^-)+\dim_\C(\FF^-)+\dim_\C(\GG^-)\\
&= \ind(h|_{\VV\times\VV}) + \ind(h|_{\VV^h\times\VV^h}) + \dim_\C(\VV\cap\VV^h),
\end{split}
\end{equation}
which is the identity that we wanted to prove.

Let us now relax the assumption that $h$ is non-degenerate, and call $\K:=\ker(h)=\ker(H)$. The form $h$ induces a non-degenerate bilinear form $h'$ on the quotient $\C^d/\K$ simply by
\begin{align*}
h'(v+\K,w+\K)=h(v,w).
\end{align*}
Any vector subspace of $\C^d/\K$ is of the form $\VV/\K$, for some vector subspace $\VV\subseteq\C^d$, and this correspondence behaves naturally with respect to the passage to the $h$-orthogonal, i.e. 
\[(\VV/\K)^{h'}=\VV^h/\K.\] 
By applying~\eqref{e:restricted_index_nondegenerate_case} to the non-degenerate Hermitian form $h'$, we obtain
\begin{align*}
\ind(h') = \ind(h'|_{\VV/\K\times\VV/\K}) + \ind(h'|_{\VV^h/\K\times\VV^h/\K}) + \dim_\C(\VV/\K\cap\VV^h/\K).
\end{align*}
Notice that 
\begin{align*}
\ind(h')&=\ind(h),\\ 
\ind(h'|_{\VV/\K\times\VV/\K})&=\ind(h|_{\VV\times\VV}),\\
\ind(h'|_{\VV^h/\K\times\VV^h/\K})&=\ind(h|_{\VV^h\times\VV^h}).
\end{align*}
Finally, since $\K=\ker(H)\subset\VV^h$,
\begin{align*}
\dim_\C(\VV/\K\cap\VV^h/\K)=\dim_\C((\VV\cap\VV^h)/\K)=\dim_\C(\VV\cap\VV^h)-\dim_\C(\VV\cap\ker(H)).
\end{align*}
This completes the proof.
\end{proof}

\subsection{Generalized eigenspaces of symplectic matrices}

Consider the standard symplectic vector space $(\R^{2d},\omega)$. The symplectic form $\omega$ can be extended to a non-degenerate skew-Hermitian form on $\C^{2d}$ by setting
\begin{align*}
\omega(\lambda z,z')=\lambda\,\omega(z,z')=\omega(z,\overline\lambda z'),\qquad\forall z,z'\in\R^{2d},\ \lambda\in\C.
\end{align*}
We denote by $\Sp(2d,\C)$ the complex symplectic group, which is given by the $2d\times2d$ complex matrices $P$ such that 
\begin{align*}
\omega(z,z')=\omega(Pz,Pz'),\qquad\forall z,z'\in\C^{2d}.
\end{align*}
Notice that  $\Sp(2d,\C)$ contains the real symplectic group $\Sp(2d)=\Sp(2d,\R)$, which is given by the matrices as above having zero imaginary part. Given a complex symplectic matrix $P\in\Sp(2d,\C)$, we are interested in its generalized eigenspaces
\begin{align*}
\FF_\lambda := \ker(P-\lambda I)^{2d},\qquad\lambda\in\C.
\end{align*}
Notice that the generalized eigenspaces span $\C^{2d}$, i.e.
\begin{align*}
 \C^{2d} = \bigoplus_{\lambda\in\sigma(P)} \FF_\lambda.
\end{align*}

\begin{lem}\label{l:omega_orthogonal_eigenspaces}
Given a pair of eigenvalues $\lambda,\theta\in\C$ of a complex symplectic matrix $P\in\Sp(2d,\C)$ such that $\lambda\overline\theta\neq 1$, the generalized eigenspaces $\FF_\lambda$ and $\FF_\theta$ are $\omega$-orthogonal, i.e.\ $\omega(z,z')=0$ for all $z\in\FF_\lambda$ and $z'\in\FF_\theta$.
\end{lem}

\begin{proof}
Consider two arbitrary generalized eigenvectors $z\in\FF_\lambda$ and $z'\in\FF_\theta$. We will prove the lemma by induction on the sum of the ranks of $z$ and $z'$. If $(P-\lambda I)^n z=(P-\theta I)^mz'=0$ with $n=m=1$, we have
\begin{align*}
\omega(z,z')=\omega(Pz,Pz')=\lambda\overline\theta\,\omega(z,z'),
\end{align*}
which implies that $\omega(z,z')=0$ since $\lambda\overline\theta\neq1$. Let us make the inductive hypothesis that $\omega(z,z')=0$ holds whenever $n+ m\leq k$. 

Consider $z$ and $z'$ such that $n+ m=k+1$, and set $w:=(P-\lambda I)z$ and $w':=(P-\theta I)z'$. The generalized eigenvectors $w$ and $w'$ have rank $n-1$ and $m-1$ respectively. By the inductive hypothesis, we have 
\begin{align*}
\omega(z,w')=\omega(w,z')=\omega(w,w')=0,
\end{align*}
which implies
\begin{align*}
\omega(z,Pz') & = \overline\theta\,\omega(z,z'),\\
\omega(Pz,z') & = \lambda\,\omega(z,z'),
\end{align*}
and
\begin{align*}
\omega(Pz,Pz') &= \lambda\,\omega(z,Pz') + \overline\theta\,\omega(Pz,z') - \lambda\overline\theta\,\omega(z,z')=\lambda\overline\theta\,\omega(z,z').
\end{align*}
Since $P$ is a symplectic matrix, this latter equality becomes $\omega(z,z')=\lambda\overline\theta\,\omega(z,z')$, and as before this implies $\omega(z,z')=0$.
\end{proof}

Consider now a real symplectic matrix $P\in\Sp(2d)$, and the real generalized eigenspace
\begin{align*}
\EE_1:=\ker(P-I)^{2d}\subset\R^{2d}.
\end{align*}

\begin{lem}\label{l:EE_1_is_symplectic}
The space $\EE_1$ is a (possibly zero dimensional) symplectic vector subspace of $(\R^{2d},\omega)$.
\end{lem}

\begin{proof}
Consider the complex generalized eigenspaces of $P$, which give the direct sum decomposition $\C^{2d}=\FF_1\oplus\FF'$, where
\begin{align*}
\FF'=\bigoplus_{\lambda\neq1} \FF_\lambda.
\end{align*}
By Lemma~\ref{l:omega_orthogonal_eigenspaces}, the vector subspaces $\FF_1$ and $\FF'$ are $\omega$-orthogonal. Since $\omega$ is a non-degenerate skew-Hermitian form on $\C^{2d}$, this implies that its restriction to $\FF_1$ is non-degenerate. Since $\omega$ is a real bilinear form, its restriction to the real part of $\FF_1$ must be non-degenerate as well. But the real part of  $\FF_1$ is precisely $\EE_1$.
\end{proof}

\bibliography{_biblio}
\bibliographystyle{amsalpha}

\end{document}